\documentclass[reqno, a4paper]{amsart}
\usepackage{amsmath, amssymb,amsthm}

\pdfoutput=1
\usepackage{color}
\def\comment#1{}

\usepackage{hyperref}
\usepackage{txfonts}
\usepackage{enumerate}
\usepackage{caption}
\usepackage{accents}

\newcommand{\OB}{Ob{\l}{\'o}j}

\newtheorem{theorem}{Theorem}
\newtheorem{assumption}[theorem]{Assumption}
\newtheorem{corollary}[theorem]{Corollary}
\newtheorem{definition}[theorem]{Definition}
\newtheorem{lemma}[theorem]{Lemma}
\newtheorem{proposition}[theorem]{Proposition}

{Important Convention}

\theoremstyle{remark}
\newtheorem{remark}[theorem]{Remark}
\newtheorem{example}[theorem]{Example}

  \newcommand{\F}{\mathcal{F}}
 
 \newcommand{\M}{\mathsf{M}}

 \newcommand{\CCC}{\mathcal{C}}

 \renewcommand{\phi}{\varphi}
\newcommand{\pp}{\mathcal P}

\newcommand{\E}{\mathbb{E}}

\newcommand{\N}{\mathbb{N}}

\newcommand{\R}{\mathbb{R}}

\newcommand{\law}{\text{\it{law}}\,}%{\mathrm{\emph{law}}}

\newcommand{\bes}{\begin{subequations}}
\newcommand{\ees}{\end{subequations}}
\newcommand{\eea}{\end{eqnarray}}

\newcommand{\EE}{{\mathbb E}}

\newcommand{\cal}{\mathcal}

\newcommand{\QQ}{{\mathbb Q}}

\usepackage{bbm}

\renewcommand{\epsilon}{\varepsilon}

\newcommand{\fourIdx}[5]{%
\setbox1=\hbox{\ensuremath{^{#1}}}%
 \setbox2=\hbox{\ensuremath{_{#2}}}%
 \setbox5=\hbox{\ensuremath{#5}}%
 \hspace{\ifnum\wd1>\wd2\wd1\else\wd2\fi}%
 \ensuremath{\copy5^{\hspace{-\wd1}\hspace{-\wd5}#1\hspace{\wd5}#3}%
 _{\hspace{-\wd2}\hspace{-\wd5}#2\hspace{\wd5}#4}%
 }}

\numberwithin{equation}{section}
\numberwithin{theorem}{section}

\renewcommand{\subset}{\subseteq}

\newcommand{\bra}[1]{\left( #1 \right)}

\usepackage{subcaption}
\usepackage{graphicx}
\newcommand{\sbm}{sBm}
\newcommand{\ssbm}{s$^2$Bm}
\newcommand{\leqc}{\preceq_c}

\begin{document}

\title{Martingale Benamou--Brenier: a probabilistic perspective}
\author{J. Backhoff-Veraguas}
\author{M. Beiglb\"ock}\author{M.
  Huesmann} \author{S. K\"allblad}  \thanks{MB gratefully acknowledges support by FWF-grant Y00782. MH has been partially supported by the Deutsche Forschungsgemeinschaft through the CRC 1060 {\it The Mathematics of Emergent Effects} and the {\it Hausdorff Center for Mathematics} and MH has been partially funded by the Vienna Science and Technology Fund (WWTF) through project VRG17-005.}  
  \begin{abstract}
 
In classical optimal transport, the contributions of  Benamou--Brenier and McCann  regarding the time-dependent version of the problem are cornerstones of the field and form the basis for   a variety of applications in other mathematical areas.

We suggest a Benamou--Brenier type formulation of the martingale transport problem for given $d$-dimensional distributions $\mu, \nu $ in convex order.
The unique solution $M^*=(M_t^*)_{t\in [0,1]}$ of this problem turns out to be a Markov-martingale which has several notable properties: In a specific sense it mimics the movement of a Brownian particle as closely as possible subject to the conditions $M^*_0\sim\mu, M^*_1\sim \nu$. Similar to McCann's displacement-interpolation, $M^*$ provides a time-consistent interpolation between $\mu$ and $\nu$. For particular choices of the initial and terminal law, $M^*$ recovers archetypical martingales such as Brownian motion, geometric Brownian motion, and the Bass martingale. Furthermore, it yields a natural  approximation to  the local vol model and a new approach to Kellerer's theorem.

This article is parallel to the work of Huesmann-Trevisan, who consider a  related class of  problems  from a PDE-oriented perspective.

\medskip

\noindent\emph{Keywords:} Optimal Transport, Martingales, weak transport problems, Brenier's Theorem, Benamou-Brenier, cyclical monotonicity, causal transport, Knothe Rosenblatt coupling, Schr\"odinger problem. \\
\emph{Mathematics Subject Classification (2010):} Primary 60G42, 60G44; Secondary 91G20.
\end{abstract}
\maketitle

\section{Introduction}
The roots of  optimal transport as a mathematical field go  back to Monge \cite{Mo81} and Kantorovich \cite{Ka42} who established its modern formulation.  Important triggers for its steep development in the last decades were the seminal results of 
Benamou, Brenier, and McCann \cite{Br87, Br91, BeBr99, Mc94}. Today the field is famous for its striking applications in areas ranging from mathematical physics and PDE-theory to geometric and functional inequalities. We refer to \cite{Vi03, Vi09, AmGi13, Sa15} for comprehensive accounts of the theory.

Recently there has also been  interest in optimal transport problems where the transport plan must satisfy additional martingale constraints. Such problems arise naturally in robust finance, but are also of independent mathematical interest, for example they have important consequences for the study of martingale inequalities (see e.g.\ \cite{BoNu13,HeObSpTo12,ObSp14}) and the Skorokhod embedding problem \cite{BeCoHu14, KaTaTo15}.  Early papers to investigate such problems include \cite{HoNe12, BeHePe12, TaTo13, GaHeTo13, DoSo12, CaLaMa14}, and this topic is commonly referred to as martingale optimal transport.

In view of the central role taken by the seminal  results of Benamou, Brenier, and McCann on  optimal transport for  squared Euclidean distance, the related  continuous time transport problem and McCann's displacement interpolation, it is intriguing to search for similar concepts also in the martingale context. While \cite{BeJu16, HeTo13} propose a martingale version of Brenier's monotone transport mapping, our starting point is the Benamou-Brenier continuous time transport problem which we  restate here for  comparison with the martingale analogues that we will consider subsequently.  

\subsection{Benamou-Brenier transport problem and McCann-interpolation in probabilistic terms}

In view of the probabilistic  nature of the results we present subsequently, it is convenient to recall  some classical  concepts and results of optimal transport in probabilistic language.
Given probabilities $\mu, \nu $ in  the space $\pp_2(\R^d) $ of $d$-dimensional distributions with finite second moment 
consider 
\begin{align}\label{MBBB}\tag{BB}
T_2(\mu, \nu):=\inf_{X_t=X_0+\int_0^tv_s\, ds, X_0\sim \mu,X_1\sim \nu} {\mathbb E}\left[\int_0^1 |v_t|^2\, dt\right].
\end{align}
Then by \cite{Br87} we have

\begin{theorem}\label{BrenierTheoremIntro} Let $\mu, \nu \in \pp_2(\R^d)$ and assume that $\mu$ is absolutely continuous with respect to Lebesgue measure. Then \eqref{MBBB} has a unique optimizer $X^*$.
\end{theorem}
\begin{remark}\label{UniquenessMeaning} \comment{MB: Potentially we want to elaborate a bit more in this remark}
In Theorem \ref{BrenierTheoremIntro}  (and similarly below) the solution to \eqref{MBBB} is unique in the sense that there exists a unique probability measure on the pathspace $C([0,1])$ such that the canonical/identity process optimizes \eqref{MBBB}.
\end{remark}

In probabilistic terms, McCann's displacement interpolation can be  defined by   $[\mu,\nu]_t:= \law(X^*_t)$ where $t\in [0,1]$ and $\mu, \nu$, $X^*$ are as in Theorem \ref{BrenierTheoremIntro}. 
\begin{theorem}\label{DisplacementIntro}
Let $\mu, \nu \in \pp_2(\R^d)$ and assume that $\mu$ is absolutely continuous with respect to Lebesgue measure. Let $s, t, \lambda\in [0,1], s < t$. Then 
\begin{align}\left[[\mu, \nu]_s, [\mu, \nu]_t\right]_\lambda =[\mu, \nu]_{(1-\lambda) s + \lambda t }.\end{align} 
Moreover
\begin{align} (t-s)\,  T_2^{1/2}\,  (\mu, \nu)= T_2^{1/2} ( [\mu, \nu]_s, [\mu, \nu]_t).\end{align}
\end{theorem}
Finally, the optimizer of \eqref{MBBB} is given through the gradient of a convex function. More precisely, by \cite{BeBr99}, we have
\begin{theorem}\label{BrenierStructureTheoremIntro} Assume that $\mu$ is absolutely continuous with respect to Lebesgue measure and $\mu, \nu \in \pp_2(\R^d)$.  A candidate process $X$, $X_0\sim \mu, X_1\sim \nu$ is an optimizer if and only if $X_1= f(X_0)$, where $f$ is the gradient of a convex function $\phi:{\mathbb R}^n\to {\mathbb R}$ and all particles move with constant speed, i.e.\ $X_t= t X_1 + (1-t) X_0 = X_0 + t(X_1-X_0)$. 
\end{theorem}

\subsection{Martingale counterparts}

Let $\mu, \nu\in \pp_2(\R^d)$ be in convex order  (denoted $\mu\leqc \nu$) and write $B$ for Brownian motion on $\R^d$. 
We consider the optimization problem 
\begin{align}\label{MBMBB}\tag{MBB}
MT(\mu, \nu):=\sup_{\substack{M_t=M_0+\int_0^t  \sigma_s\, dB_s \\ M_0\sim \mu,M_1\sim \nu}} {\mathbb E}\left[\int_0^1 \mbox{tr}(\sigma_t)\, dt\right],
\end{align} 
see also \eqref{MBMBBv2} below.
We have

\begin{theorem}\label{MainTheoremIntro} Assume that $\mu, \nu \in \pp_2(\R^d)$ satisfy $\mu\leqc\nu$. Then \eqref{MBMBB} has an optimizer $M^*$ which is unique in law.
\end{theorem}

At its face, the optimization problems \eqref{MBBB} and \eqref{MBMBB} look rather different. However it is not hard to see that both problems are equivalent to optimization problems that are much more obviously related. In Section \ref{sec causal} below we establish that
\begin{align}
X^*&=\text{argmin}_{X_0\sim\mu, X_1\sim\nu} W^2(X, \text{constant speed particle}),\\ 
M^*&=\text{argmin}_{M_0\sim\mu, M_1\sim\nu} W^2_c(M, \text{constant volatility martingale}),\label{CausalFormulation}
\end{align}
where $W^2$ denotes Wasserstein distance with respect to squared Cameron-Martin norm, while $W^2_c$ denotes an \emph{adapted} or  \emph{causal} analogue\footnote{Causal transport plans generalize adapted processes in the same way as classical Kantorovich transport plans extend Monge maps.} (in the terminology of Lassalle \cite{La13}), see Section \ref{sec causal} for details.

The reformulation in  \eqref{CausalFormulation} allows for the following interpretation: $M^*$ is the process whose evolution follows the movement of a Brownian particle as closely as possible subject to the marginal conditions $M_0\sim \mu, M_1\sim \nu$. This motivates the name in the following definition.
\begin{definition} \label{def d dim sBm}
Let $\mu, \nu, M^*$ be as in Theorem  \ref{MainTheoremIntro}. Then we call $M^*$ the \emph{stretched Brownian motion} (\sbm) from $\mu$ to $\nu$. We define the \emph{martingale displacement interpolation} by 
\begin{align}
[\mu, \nu]^M_t:= \law M^*_t\,,
\end{align}
for $t\in [0,1]$. 
\end{definition}

In analogy to Theorem \ref{DisplacementIntro} we have
\begin{theorem}\label{MDisplacementIntro}
Assume that $\mu, \nu \in \pp_2(\R^d)$ satisfy $\mu\leqc\nu$. Let $s, t, \lambda\in [0,1], s < t$. Then 
\begin{align}\left[[\mu, \nu]^M_s, [\mu, \nu]^M_t\right]^M_\lambda =[\mu, \nu]^M_{(1-\lambda) s + \lambda t }\,.\end{align} 
Moreover
\begin{align} (t-s)\,  MT^{2}\,  (\mu, \nu)= MT^{2} ( [\mu, \nu]^M_s, [\mu, \nu]^M_t).\end{align}
\end{theorem}

\subsection{Structure of stretched Brownian motion}

In the solution of the  classical Benamou--Brenier transport problem, particles travel with constant speed along straight lines. In contrast, we will see that in the case of \sbm\ the movement of individual particles mimic that of Brownian motion. Broadly speaking, the ``direction'' of these particles will be  determined -- similar to the classical case -- by a mapping which is the gradient of a convex function. 

\medskip

For simplicity, we first consider the particular case where $\mu, \nu, \mu\leqc \nu$ are probabilities on the real line and $\mu$ is concentrated in a single point, i.e.\ $\mu= \delta_m$ where $m$ is the center of $\nu$. It turns out that in this case \sbm\ $M^*$ is precisely the ``Bass martingale'' \cite{Ba83} (or `Brownian martingale') with terminal distribution $\nu$. We briefly recall its construction: 
Pick $f:{\mathbb R}\to{\mathbb R}$ increasing such that $f(\gamma)=\nu$, where $\gamma$ is the standard Gaussian distribution on $\R$. 
Then set for $t\in [0,1]$
\begin{align}\label{MBBass}
M_t:= \E[f(B_1)|\F_t]= \E[f(B_1)|B_t] = f_t(B_t),
\end{align}
where $B=(B_t)_{t\in [0,1]}$ denotes Brownian motion started in $B_0\sim \delta_{0}$, $(\F_t)_{t\in [0,1]}$ the Brownian filtration and $f_t(b):=\int f(b+y)\, d\gamma_{1-t}(y)$, $\gamma_s\sim N(0,s)$. 
Clearly $M$ is a continuous Markov martingale such that $M_0\sim \delta_m, M_1\sim \nu$. As a particular consequence of the results below we will see that $M$ is a stretched Brownian motion.

\medskip

To state our results for the general, multidimensional case we need to consider an extension of the Bass construction. Let $F: \R^d\to \R$ be a convex function and set \begin{align} \textstyle \label{eq f_t}
f_t(b)=\int \nabla F(b + y)\gamma_{1-t}^d(dy),
\end{align} where $\gamma_s^d$ denotes the centered $d$-dimensional Gaussian with covariance matrix $s\, \mbox{Id}$.\comment{MB: Any way we can write this in a less clumsy fashion?}  If $B$ denotes $d$-dimensional Brownian motion started in $B_0 \sim \alpha$, we have 
\begin{align} \E[\nabla F(B_1)| \F_t]= f_t(B_t), \ t\in [0,1]. 
\end{align}
\begin{definition}\label{def d dim ssBm}
A continuous $\R^d$-valued martingale $M$ is a \emph{standard stretched  Brownian motion (\ssbm)} from $\mu$ to $\nu$ if there exist a probability measure $\alpha$ on $\R^d$ and a convex function $F:{\mathbb R}^d\to {\mathbb R}$ with $\nabla F(\alpha*\gamma^d)=\nu$, such that $$M_t = E[\nabla F(B_1)| \F_t]\, \mbox{  and  }\, M_0\sim \mu,$$ 
where $B$ is a Brownian motion with $B_0\sim \alpha$. 
\end{definition}

Note, that for $\alpha,\nu\in \pp_2(\R^d)$ there exists a convex function $F$ with $\nabla F(\alpha*\gamma^d)=\nu$  and $F$ is $\alpha*\gamma^d$-unique up to an additive constant. (This is a consequence of Brenier's Theorem, see e.g.\ Theorem \ref{BrenierTheoremIntro} or \cite[Theorem 2.12]{Vi03}.) 

\begin{remark}
 Both Brownian motion and geometric Brownian motion are examples of standard stretched Brownian motion. 
\end{remark}

We have the following results
\begin{theorem}\label{ThmStandardtoStretchedIntro} Let $\mu, \nu \in \pp_2(\R^d)$ with $ \mu\leqc \nu$. If $M$ is a standard stretched Brownian motion from $\mu $ to $\nu$, then $M$ is an optimizer of \eqref{MBMBB}, i.e.\ $M$ is the stretched Brownian motion from $\mu$ to $\nu$. 
\end{theorem}

\begin{theorem}\label{ThmStretchedtoStandardIntro}Let $\mu, \nu \in \pp_2(\R^d)$ with $ \mu\leqc \nu$. Let  $M^*$ be the stretched Brownian motion from $\mu $ to $\nu$, i.e.\ the optimizer of \eqref{MBMBB}. 
Write $M^{*,x}$ for the martingale $M$ conditioned on starting in $M_0= x$. 
Then for $\mu$-a.a.\ $x\in \R^d$ the martingale $M^{*,x}$ is a standard stretched Brownian motion. 
\end{theorem}

As a particular consequence of these results, the notions \sbm\ and \ssbm\ coincide if $\mu$ is concentrated in a single point. However the relation between  \sbm\ and \ssbm\ is more complicated in general: 
A notable intricacy of the martingale transport problem is caused by the fact that, loosely speaking, certain regions of the space do not communicate with each other. 

Consider for a moment the particular case where $\mu, \nu$ are distributions on the real line. In this instance,  a martingale transport problem can be decomposed into countably many ``minimal'' components and on each of these components the behaviour of the problem is fairly similar to the classical transport problem. We refer the reader to Section \ref{sec main dim 1} for the precise definition and only provide an illustrative example at this stage. 

\begin{example} Let $\mu:=1/2 (\lambda_{|[-3,-2]} + \lambda_{|[2,3]}),$ $\nu:=1/6 (\lambda_{|[-4,-1]} + \lambda_{|[1,4]})$. Then \emph{any} martingale $M, M_0 \sim \mu, M_1\sim\nu$ will satisfy the following: If $M_0 >0$, then $M_1>0$ and if $M_0 \leq0 $ then $M_1\leq 0$. I.e.\ the positive and the negative halfline do not ``communicate,'' and a problem of  martingale transport should be considered on either of these parts of space separately. 
\end{example}

 If the pair $(\mu, \nu)$ decomposes into more than one minimal component, as in the previous example,  there exists no \ssbm\ from $\mu$ to $\nu$. However for the one-dimensional case  we will establish the following:  A martingale is a \sbm\ if and only if it behaves like a \ssbm\ on each minimal component, see Theorem \ref{MainTheoremOneDim}.  

\medskip

Notably, the challenges posed by non-communicating regions appear much more intricate for dimension $d\geq 2$, see the deep contributions  of Ghoussoub--Kim--Lim \cite{GhKiLi16},  DeMarch--Touzi \cite{DMTo17} and \OB--Siorpaes \cite{ObSi17}. In particular it is not yet fully understood how to break up a martingale transport problem into distinct pieces which mimic the behaviour of minimal components in the one dimensional case. 

 Below we will give special emphasis to the case $d=2$ under the additional regularity assumption that $\nu$ is absolutely continuous. This instance seems of particular interest since it allows to recognize the geometric structure of the problem while avoiding the more intricate effects of non-minimality which are present in higher dimension. Based on  the results of  \cite{DMTo17,ObSi17} and a particular `monotonicity principle' we will be able to largely recover the main one-dimensional result (Theorem  \ref{MainTheoremOneDim}) in the two-dimensional case, see Sections \ref{sec preliminaries}-\ref{sec main dim 2} below and specifically Theorem \ref{thm main} therein.
 We conjecture that a similar structural characterization  of \sbm\ can be established in general dimensions, pending future developments in the direction of \cite{DMTo17,ObSi17}.

\subsection{Further remarks}

\subsubsection{Discrete time version and monotonicity principle} 

The classical Benamou--Brenier transport formulation immediately reduces to the familiar discrete time transport problem for squared  distance costs. Similarly,  the martingale version \eqref{MBMBB} can be reformulated as discrete time problem, more precisely, a weak transport problem in the sense of \cite{GoRoSaTe14}. 

The discrete time reformulation of  \eqref{MBMBB} plays an important role in the derivation of our main results. To analyze the discrete problem we introduce a ``monotonicity principle'' for weak transport problems. The origin of this approach is the characterization of optimal transport plans in terms of $c$-cyclical monotonicity. In optimal transport, the potential of this concept has been recognized by Gangbo--McCann \cite{GaMc96}. More recently, variants of this idea have proved to be useful in a number of related situation, see \cite{KiPa13, BeGr14, Za14,  GuTaTo15b,  BeJu16, BeCoHu14, NuSt16, BeEdElSc16} among others. In view of this, it seems possible that the  monotonicity principle for weak transport problems could also be of interest in its own right {(cf.\ \cite{GoJu18, BaBePa18} which appeared after we first posted this article). }

\subsubsection{Schr\"odinger problem}
Our variational problem \eqref{MBMBB} is reminiscent of the celebrated Schr\"odinger problem, in which the idea is to minimize the relative entropy with respect to Wiener measure (or other Markov laws) over path-measures with fixed initial and final marginals. We refer to the survey \cite{Le14} and the references therein. Among the similarities, let us mention that the solution to the Schr\"odinger problem is unique and is a Markov law, and furthermore this problem also has a transport-like discrete time reformulation which is fundamental to the dynamic path-space version. On the other hand, \eqref{MBMBB} and the Schr\"odinger problem are in particular sense at opposing ends of probabilistic variational problems, we optimize over ``volatilities keeping the drift fixed'' whereas the latter optimizes over ``drifts keeping the volatility fixed.''

\subsubsection{Bass-martingale and Skorokhod embedding} The Bass-martingale \eqref{MBBass} was used by Bass \cite{Ba83} to solve the Skorokhod embedding problem. Hobson asked whether there are natural optimality properties related to this construction and if one could give a version with a non trivial starting law. \eqref{MBMBB} yields such an optimality property of the Bass construction and  stretched Brownian motion gives rise to a version of the Bass embedding with non trivial starting law. Notably a characterization of the Bass martingale in terms of an optimality property was first obtained in \cite{BeCoHuKa17}, the variational problem considered in that article refers to measure valued martingales and appears rather different from the one considered in \eqref{MBMBB}.

\subsubsection{Geometric Brownian motion.} From the above results it is clear that Brownian motion is (up to an  appropriate scaling of time) a \ssbm\  between any of its marginals. In fact, the same holds for Brownian martingales $dM_t=\sigma dB_t$ for constant and time-independent $\sigma$. We find it notable that same applies in the case of  \emph{geometric} Brownian motion. 

\subsubsection{Kellerer's theorem and Lipschitz kernels} 
 Kellerer's theorem \cite{Ke73} states that if a family of distributions $(\mu_t)_{t\in[0,1]}$ on the real line satisfies 
$s\leq t\Rightarrow \mu_s\leqc \mu_t,$
there exists a Markovian martingale $(X_t)_{t\in \R_+}$ with $\law(X_t)=\mu_t$ for every $t$. In contemporary terms (see \cite{HPRY}),  $(\mu_t)_{t\in \R_+}$ is called a \emph{peacock} and $(X_t)_{t\in \R_+}$ is a Markovian martingale associated to this peacock.

The technically most involved part in establishing Kellerer's theorem is to prove that for $\mu\leqc\nu$ there exists a martingale transition kernel $P$ having the following \emph{Lipschitz-property}: A kernel $P: x\mapsto \pi_x $, $\nu(dy)= \int \mu(dx) \pi_x(dy)$ is called Lipschitz (or more precisely $1$-Lipschitz) if  $\mathcal{W}_1(\pi_x,\pi_{x'})\leq |x-x'|$ for all $x,x'$. Kellerer's proof of the existence  of Lipschitz-kernels is not constructive and employs Choquet's theorem. Other proofs are based on solutions to the Skorokhod problem for non-trivial starting law, see \cite{{Lo08b}, BeHuSt16}. 
 
Stretched Brownian motion yields a new construction of a Lipschitz-kernel: Given probabilities $\mu, \nu, \mu\leqc \nu$ on the real line and writing $M^*$ for \sbm\ from $\mu$ to $\nu$, then  $\law (M_1^*| M_0^*)$ is a Lipschitz kernel. We provide the argument in Corollary \ref{coro LM} below.

The question whether Kellerer's theorem can be extended to the case of marginal measures on $\R^d, d\geq 2$ remains open. While all previously known constructions of kernels used for the proof of Kellerer's theorem were inherently limited to dimension $d=1$, the approach sketched above seems more susceptible to generalization. We intend to pursue this question further in future work. 

\subsubsection{Almost continuous diffusions / local volatility model}\label{LocalVolParagraph} Assume that $(\mu_t)_{t\in [0,1]}$ (where $\mu_t, t\in [0,1]$ are probabilities on the real line)  is a peacock such that $t\mapsto \mu_t$ is continuous in the weak topology.   Lowther \cite{Lo08b} establishes that an appropriate continuity condition makes the Markov martingale appearing in Kellerer's theorem unique. In his terms, there is a unique ``almost continuous'' martingale diffusion $M^{ac}$ such that $M^{ac}_t\sim \mu_t, t\in [0,1]$. Under further regularity conditions, $M^{ac}$ is precisely Dupire's local volatility model. 

Stretched Brownian motion yields a simple approximation scheme to $M^{ac}$.  
 Write $M^n$ for the Markov martingale satisfying that  for each $k\in \{1,\ldots, n\}$, $(M^n_t)_{t\in [(k-1)/n, k/n]}$ is (modulo the obvious affine time-change) stretched Brownian motion between $\mu_{(k-1)/n}$ and $\mu_{k/n}$. $M^n$ is then a  continuous diffusion and based on Lowther's \cite{Lo08b, Lo09} it is straightforward that  
 \begin{align}\label{LocVolLimit}
 M^{ac}= \lim_{n\to \infty} M^n, 
\end{align}
 where the limit is in the sense of convergence of finite dimensional distribution (cf.\ \cite{BeHuSt16}).

\subsubsection{L\'evy processes}
Many arguments in this article rely only on the independence and stationarity of increments of Brownian motion. Therefore a problem similar to \eqref{MBMBB}, but based on a reference L\'evy process instead, should conceivably exhibit similar properties as we find in the Brownian case. In this direction it could be an interesting question to identify the outcome of the approximation procedure described in \eqref{LocVolLimit}.

\subsubsection{Dual problem, related work} Optimization problems similar to \eqref{MBMBB} were first studied from a general perspective by Tan and Touzi \cite{TaTo13}, in particular establishing a duality theory for these type of problems. The dual viewpoint is also emphasized in \cite{HuTr17}, which is parallel to the present work. Among other results, \cite{HuTr17} derives a PDE that yields a sufficient condition for a flow of measures to optimize  \eqref{MBMBB} or related cost criteria.

\subsection{Outline of the article: } In Section  \ref{sec refined} we introduce the discrete-time variant of our optimization problem. We also prove some of the multidimensional results stated in the introduction and provide  further properties of \sbm\  (dynamic programming principle for \eqref{MBMBB}, the Markov property of \sbm). In Section \ref{sec main dim 1 and 2} we state our main results regarding the structure of \sbm\ in dimensions one and two. In Section \ref{sec mono} we present a monotonicity principle for weak transport problems, which is crucial for our analysis in dimension two, but may also be of  independent interest. In Section \ref{sec pending proofs} we conclude the proofs of our main results.  Finally in Section \ref{sec causal} we present further optimality properties of \sbm\ and \ssbm\ in terms of a (causal) optimal transport problem between martingale laws.

\subsection{Notation:} 
The set of probability measures on a set $\mathsf{X}$ will be denoted by $\mathcal P(\mathsf{X})$. For $\rho_1,\rho_2\in\mathcal P(\mathsf{X})$ we write $\Pi(\rho_1,\rho_2)$ for the set of all couplings of $\rho_1$ and $\rho_2$, i.e.\ all measures on the product space with marginals $\rho_1$ and $\rho_2$ resp.
 Two probability measures $\mu, \nu \in \mathcal P(\R^d)$ are said to be in convex order, short $\mu\leqc\nu$ iff for all convex real valued functions $\phi$ it holds that $\int \phi \, d\mu\leq \int \phi \, d\nu.$ 
\medskip\\
\textit{In this article, we fix  $\mu,\nu\in\mathcal P(\R^d)$, assume that $\mu\leqc\nu$ and that both measures have finite second moment. }
\medskip\\
 We denote by $\M (\mu,\nu)$ the set of all martingale couplings with  marginals $\mu$ and $\nu$ (which is non-empty by Strassen's Theorem \cite{St65}), i.e.\
$$\M (\mu,\nu):=\{\pi\in\mathcal{P}(\R^d\times \R^d): \E^\pi[(y-x)h(x)]=0\mbox{ for all $h:\R^d\to\R$ Borel bounded}\}.$$ 
For a generic measure $\pi$ on $\R^d\times\R^d$ we denote by $(\pi_x)_{x\in\R^d}$ the conditional transition kernel given the first coordinate or equivalently its disintegration w.r.t.\ the first marginal. For $\rho\in\mathcal P(\mathsf{X})$ and a measurable map $f:\mathsf{X}\to\mathsf{Y}$ we write $f(\rho)=\rho\circ f^{-1}$ for the pushforward of $\rho$ under $f$.

For a set $A\subset \R^d$ we denote by $\mbox{aff}(A)$ the smallest affine vector space containing it, $\mbox{dim}(A)$ the dimension of $\mbox{aff}(A)$, $\mbox{ri}(A)$  the relative interior  of $A$ (i.e.\ interior of $A$ with respect to the relative topology of $\mbox{aff}(A)$ as inherited from the usual topology in $\R^d$), and $\partial A:= \overline{A}\backslash \text{ri}(A)$ the relative boundary. By $\mbox{co}(A)$ and $\overline{\mbox{co}}(A)$ we denote the convex hull and the closed convex hull of $A$ respectively. The relative face of $A$ at $a$ is defined by $\text{rf}_a(A)=\{y\in A: (a-\varepsilon(y-a),y+\varepsilon(y-a))\subset A,\text{ some }\varepsilon >0 \}$. For a set $\Gamma\subset \R^d\times\R^d$ we denote $\Gamma_x:=\{y:(x,y)\in\Gamma\}$ and $\mbox{proj}_1(\Gamma)$ the projection of $\Gamma$ onto the first coordinate. Given $\pi\in \M (\mu,\nu)$ we say that $\Gamma\subset \R^d\times\R^d$ is a martingale support for $\pi$ if $\pi(\Gamma)=1$ and $x\in \mbox{ri}(\Gamma_x)$ for $\mu$-a.e.\ $x$. 

Finally, we denote by $\lambda^d,\gamma^d,\gamma^d_t$ resp.\ the Lebesgue, standard Gaussian, and the Gaussian measure with covariance matrix $t\,\text{Id}$ in $\R^d$, and reserve the symbol $*$ for convolution.

\section{Refined and auxiliary results in arbitrary dimensions}\label{sec refined}

We start by restating our main optimization problem in (slightly) more precise form.
\begin{align}\label{MBMBBv2}%\tag{MBB, v2}
%MT:=MT_{\mu,\nu}:=
MT:=MT(\mu,\nu):=\sup_{M_t=M_0+\int_0^t \sigma_s\, dB_s, M_0\sim \mu,M_1\sim \nu} {\mathbb E}\left[\int_0^1 \mbox{tr}(\sigma_t)\, dt\right].
\end{align}
Here the supremum is taken over the class of all filtered probability spaces $(\Omega,\cal F, \mathbb{P})$, with  $\sigma$ an $\R^{d\times d}$-valued $\cal F$-progressive process and $B$ a $d$-dimensional $\cal F$-Brownian motion, such that $M$ is a martingale. In fact, as a particular consequence of Theorem \ref{lem inequality static dynamic}, the choice of the underlying probability space is not relevant, provided that $(\Omega,\cal F, \mathbb{P})$ is rich enough to support a  $\F_0$-measurable random variable with continuous distribution. 

By Doob's martingale representation theorem (see e.g.\ \cite[Theorem 4.2]{KaSh91}), the supremum above is the same if we optimized over all continuous $d$-dimensional local martingales from $\mu$ to $\nu$ with absolutely continuous cross variation matrix (one then replaces the cost by the trace of the root of the Radon Nikodym density of said matrix).

We will be also interested in a ``static'' version of the above problem, just as the Benamou-Brenier formula is associated to the static  optimal transport problem with quadratic cost
\begin{align}
WT:=WT(\mu,\nu):= \sup_{\substack{\{\pi_x\}_x,\, \text{mean}(\pi_x)=x \\ \int\mu(dx)\pi_x(dy)=\nu(dy) }}\int\mu(dx) \sup_{q\in\Pi(\pi_x,\gamma^d)}\int q(dm,db)\,\, m\cdot b\, . \label{static weak transport}\tag{$WOT$}
\end{align}
\noindent The tag $(WOT)$ reflects the fact that this is a weak optimal transport problem (the cost function is non-linear in the optimization variable).

\begin{remark}\label{rem sup tp inf}
Completing the square in \eqref{static weak transport} yields
\begin{align}\label{eq weak Gozlan 2} 
1+\int|y|^2\, d\nu - 2\,WT= \inf_{\substack{\{\pi_x\}_x,\, \text{mean}(\pi_x)=x \\ \int\mu(dx)\pi_x(dy)=\nu(dy) }}\int\mu(dx) \mathcal{W}_2(\pi_x,\gamma^d)^2,
\end{align}
where $\mathcal{W}_2$ is the usual $L^2$ Wasserstein distance on $\mathcal P(\R^d)$. The r.h.s.\ of \eqref{eq weak Gozlan 2} is clearly a weak transport problem in the setting of Gozlan et.\ al.\ \cite{GoRoSaTe14,GoRoSaTe15}.
\end{remark}

\medskip

We start by establishing the link between the static and dynamic problems introduced so far, and moreover, establish the uniqueness of optimizers in either case. As a corollary, this yields Theorem \ref{MainTheoremIntro} stated in the introduction.

\begin{theorem} \label{lem inequality static dynamic}
The static and the dynamic problems \eqref{static weak transport} and \eqref{MBMBB} are equivalent. More precisely, 
\begin{enumerate}
 \item $WT=MT<\infty,$
\item \eqref{static weak transport} has a unique optimizer $\pi^*$;
\item \eqref{MBMBB} has a unique-in-law optimizer $M^*$;
\item $\pi^*=\law(M^*_0,M^*_1)$ and $M^*=G(\pi^*)$ for some function $G$, i.e.\ $M^*$ can be explicitly constructed from $\pi^*$.
\end{enumerate}
\end{theorem}

\begin{proof}
Let $M$ be feasible for \eqref{MBMBB}. By It\^o's formula and the martingale property of $M$ we have
$$ \textstyle {\mathbb E}\left[\int_0^1 \mbox{tr}(\sigma_t)\, dt\right] = {\mathbb E}[M_1\cdot B_1-M_0\cdot B_0]= {\mathbb E}[M_1\cdot (B_1-B_0)]=
{\mathbb E}[\,\, {\mathbb E}[ M_1\cdot (B_1-B_0) \, |\, M_0] \,\,] .$$
Letting $q_x = \law(M_1,B_1-B_0\,|\,M_0=x)$ we find $q_x\in\Pi(\pi_x\, ,\, \gamma^d)$ for $\pi_x = \law(M_1\,|\,M_0=x)$ and
 $$\textstyle {\mathbb E}\left[\int_0^1 \mbox{tr}(\sigma_t)\, dt\right] = \int\mu(dx)\int q_x(dm,db)~m\cdot b. $$ 
From this we easily conclude $WT \geq MT$. 

Now let $\pi$ be feasible for \eqref{static weak transport}. For each $x$ we can find $F^x(\cdot)$ convex such that $\nabla F^x(\gamma^d)=\pi_x$. We now define $M^x_t:=E[\nabla F^x(B_1)|\mathcal{F}^B_t]$ for a given standard Brownian motion on $\mathbb{R}^d$ with Brownian filtration $\F^B$. 
Potentially enlarging our probability space we can assume the existence of a random variable $X$ independent of the Brownian motion $B$ with $X\sim\mu$. We denote the filtration (on the potentially bigger probability space) by $\cal F$.
Since $M_0^x=\int y\pi_x(dy)=x$ and $\int\mu(dx)\pi_x(dy)=\nu(dy)$ we conclude that $\{M^{X}_t\}_{t\in [0,1]}$ is a continuous martingale from $\mu$ to $\nu$. By construction
$$\textstyle \int\mu(dx)  \sup_{q\in\Pi(\pi_x,\gamma^d)}\int q(dm,db)\,\, m\cdot b = \int\mu(dx) \int\gamma^d(db)\,\, b\cdot \nabla F^x(b) = {\mathbb E}\left[  \, \mathbb{E} \left[ B_1\cdot M_1^X |X  \right ]\, \right ]  , $$
and the last term equals $\textstyle\mathbb{E}[\int_0^1 \text{tr}(\sigma_t)dt ]$ as before ($\sigma$ can easily be computed from $\nabla F^x$). This proves $WT \leq MT$ and hence $WT = MT$.
The finiteness $ \infty> WT $ follows from $m\cdot b \leq |m|^2+|b|^2$ and $\nu$ and $\gamma$ having finite second moment; see \eqref{static weak transport}. 

To show that \eqref{static weak transport} is attained let us denote by $(\pi^n)_{n\in\N}$ (where $\pi^n(dx,dy)=\pi^n_x(dy)\mu(dy)$) an optimizing sequence. The set $\Pi(\mu,\nu)$ is weakly compact in $\mathcal P(\R^d\times\R^d)$. Moreover, the convex subset $\M (\mu,\nu)$ is weakly closed (hence weakly compact), e.g.\ \cite[Theorem 7,12 (iv)]{Vi03}. By \cite[Theorem 3.7]{Ba97} we obtain the existence of a measurable kernel $x\mapsto \pi_x\in\mathcal{P}(\mathbb{R}^d)$ and a subsequence, still denoted by $(\pi^n)_n$, such that on a $\mu$-full set 
$$\textstyle\frac{1}{N}\sum_{n\leq N} \pi_x^n(dy)\to \pi_x(dy),$$
with respect to weak convergence in $\mathcal{P}(\R^d)$. In particular $\frac{1}{N}\sum_{n\leq N} \pi^n\to \pi$ in the weak topology in $\mathcal{P}(\R^d\times \R^d)$,  where $\pi(dx,dy):=\mu(dx)\pi_x(dy)$\comment{MH: why do we write out $\pi$ here but not $\pi_n$ before?}. Since $\M(\mu,\nu)$ is closed,  we have that $\pi\in\M (\mu,\nu)$. Finally,
\begin{align*} 
WT& =\textstyle \lim_n \int\mu(dx) \sup_{q\in\Pi(\pi_x^n,\gamma^d)}\int q(dm,db)\,\, m\cdot b \\
& =\textstyle \lim_N \int\mu(dx) \frac{1}{N}\sum_{n\leq N}\sup_{q\in\Pi(\pi_x^n,\gamma^d)}\int q(dm,db)\,\, m\cdot b \\
& \leq\textstyle \lim_N \int\mu(dx) \sup_{q\in\Pi(  \frac{1}{N}\sum_{n\leq N}\pi_x^n,\gamma^d)}\int q(dm,db)\,\, m\cdot b  \\
& \leq\textstyle  \int\mu(dx) \limsup_N\sup_{q\in\Pi(  \frac{1}{N}\sum_{n\leq N}\pi_x^n,\gamma^d)}\int q(dm,db)\,\, m\cdot b \\
& \leq\textstyle  \int\mu(dx) \sup_{q\in\Pi( \pi_x,\gamma^d)}\int q(dm,db)\,\, m\cdot b\, \,\,\leq  WT\, .
\end{align*}
The first inequality holds by concavity of $\eta\mapsto H(\eta):= \sup_{q\in \Pi(\eta,\gamma^d)}\int q(dm,db)~m\cdot b$ w.r.t.\ convex combinations of measures. The second inequality is Fatou's lemma, noticing that the integrand is bounded in $L^1(\mu)$ (the bound equals  the sum of the second moments of $\mu$ and $\gamma$). The third inequality follows by weak convergence of the averaged kernel on a $\mu$-full set and upper semicontinuity of $H(\cdot)$. For uniqueness it suffices to notice that $H(\cdot)$ is actually strictly concave, which is an easy consequence of Brenier's Theorem. Hence, \eqref{static weak transport} is attained and we denote the unique optimizer by $\pi^*$.

Taking $\pi^*$ we may build an optimizer $M^*$ for \eqref{MBMBB} as in the first part of the proof (as the value of both problems agree).
 
We finally establish the uniqueness of optimizers for \eqref{MBMBB}. Let $\tilde M$ be any such optimizer. From the previous considerations, we deduce that the law of $(\tilde M_0,\tilde M_1)$ is the unique optimizer $\pi^*$ of \eqref{static weak transport}. Conditioning on $\{\tilde M_0=x\}$ we thus have that $\tilde M$ connects $\delta_x$ to $\pi^*_x$. It follows that $\mu(dx)$-a.s.\ $\tilde M$ conditioned on $\{\tilde M_0=x\}$ is optimal between these marginals. Indeed, 
\begin{align}\textstyle
 \sup\limits_{N_t=x+\int_0^t \sigma_s\, dB_s,\,N_1\sim \pi^*_x} {\mathbb E}\left[\int_0^1 \mbox{tr}(\sigma_t)\, dt\right]  = \sup\limits_{q\in\Pi(\pi^*_x,\gamma^d)}\int q(dm,db)\,\, m\cdot b\,\, , \label{eq aux conditioning}
\end{align}
by the results obtained so far, since if $\tilde M$ conditioned on $\{\tilde M_0=x\}$ was not optimal for the l.h.s.\ it could not deliver the equality $MT=WT$. So it suffices to show that the l.h.s.\ of \eqref{eq aux conditioning} is uniquely attained. But any candidate martingale $N$ with volatility $\sigma$ satisfies $\mathbb{E}[\int_0^1\text{tr}(\sigma_t)dt]= \mathbb{E}[ N_1 B_1]$ (since here we can assume $B_0=0$). Hence, Brenier's Theorem implies that $ \tilde M_1= \nabla F^x(B_1)$ on $\{\tilde M_0=x\}$, for a convex function $F^x$. Since the optimal transport map $\nabla F^x$ is unique, and  the martingale property determines uniquely the law of $\tilde M$, we finally get $\tilde M=M^*$ in law.
\end{proof}

\begin{remark}\label{rem connection}
The proof of Theorem \ref{lem inequality static dynamic} shows how to build the optimizer for \eqref{MBMBB} via the following procedure, making the statement $ M^*=G(\pi^*)$ in Theorem \ref{lem inequality static dynamic} (2)  precise:
\begin{enumerate}
\item Find the unique optimizer $\pi^*$ of \eqref{static weak transport}.
\item Find convex functions $F^x$ such that $\nabla F^x(\gamma^d)=\pi_x^*$.
\item Define $M^x_t := \mathbb{E}[\nabla F^x(B_1)|B_t] = \int \nabla F^x(y+B_t)\gamma^d_{1-t}(dy)$.
\item Take $X\sim \mu$ independent of $B$ and let $M_t:=M^X_t$.
\end{enumerate} 
\noindent In particular, this proves Theorem \ref{ThmStretchedtoStandardIntro} in the introduction.
\end{remark}

We now establish further properties of the optimizer $M^*$ of \eqref{MBMBB}, which hold likewise in any number of dimensions. The first two of them will be important for the proofs of the results yet to come, namely that \eqref{MBMBB} obeys a dynamic programming principle and that $M^*$  is a strong Markov martingale.{ The final property, that $M^*$ is an ``optimal constant-speed'' interpolation between its marginals, is crucial for the interpretation of our martingale as an analogue of displacement interpolation in classical transport and in particular proves Theorem \ref{MDisplacementIntro} in the introduction.}  

For  stopping times $0\leq\tau\leq T$ we define
\begin{align}\label{intermediate def}
V(\tau,T,\mu,\nu):=\sup_{\substack{M_r=X+\int_\tau^r \sigma_u\, dB_u,\, \tau\leq r\leq T \\ X\sim \mu,\,M_T\sim \nu}} {\mathbb E}\left[\int_\tau^T \mbox{tr}(\sigma_u)\, du\right],
\end{align}
so that $MT =  V(0,1,\mu,\nu)$.

\begin{lemma}[Dynamic programming principle]\label{lem DPP}
For every stopping time $0\leq\tau \leq 1$
\begin{align}\label{eq DPP}\textstyle V(0,1,\mu,\nu) =  \textstyle\sup\limits_{\substack{ M_s=M_0+\int_0^s \sigma_rdB_r \\ 0\leq s\leq \tau,\, M_0\sim\mu}}\left\{\mathbb{E}\left [ \int_0^\tau\text{tr}(\sigma_r)dr\right ]+ V(\tau,1,\law(M_\tau),\nu)  \right\},
\end{align}
with the convention that $\sup \emptyset = -\infty$. In particular if $M^*$ is optimal for $V(0,1,\mu,\nu)$, then:
\begin{enumerate}
\item $M^*|_{[\tau,1]}$ is optimal for $V(\tau,1,\law(M^*_\tau),\nu)$,
\item $M^*|_{[0,\tau]}$ is optimal for $V(0,\tau,\mu,\law(M^*_t))$
\item A.s.\ we have $\law(M^*_1|M^*_s,s\leq\tau)= \law(M^*_1|M^*_\tau)$.
\end{enumerate}
\end{lemma}

\begin{proof}
Obviously the l.h.s.\ of \eqref{eq DPP} is smaller than the r.h.s. of \eqref{eq DPP}. Take now $M^1$ feasible for the r.h.s.\ (so that $M^1$ is adapted to a filtration $\{\cal F^1_{s\wedge\tau}\}_{s\geq 0}$, $B$ is a Brownian motion on $[0,\tau]$ adapted to it, and $dM^1= \sigma^{(1)}dB$). Let $M^2$ be optimal for $V(\tau,1,\law(M^1_\tau),\nu) $. By Remark \ref{rem connection} we may build $M^2$ from the starting distribution $M^1_\tau$ and the filtration $\cal F^2$ of this random variable and a Brownian motion $W$ independent of $\cal F^1$ (and so of $M^1_\tau$), so $dM^2=\sigma^{(2)}dW$. We then build a continuous martingale $M$ on $[0,1]$ by setting it to $M^1$ on $[0,\tau]$ and $M^2$ on $(\tau,1]$, obtaining easily that $$\textstyle\mathbb{E}\left [ \int_0^\tau\text{tr}(\sigma^{(1)}_r)dr\right ]+ V(\tau,\law(M^1_\tau),\nu)=\mathbb{E}\left [ \int_0^\tau\text{tr}(\sigma^{(1)}_r)dr  + \int_\tau^1\text{tr}(\sigma^{(2)}_r)dr\right ] .$$
Observing that $\tilde{B}_s=1_{[0,\tau]}(s)B_s+ 1_{(\tau,1]}(s)[B_\tau+W_s-W_\tau]$ is a Brownian motion for the concatenation of filtrations $\cal F^1$ and $\cal F^2$, and $dM=(1_{[0,\tau]}(s)\sigma^{(1)}_s + 1_{(\tau,1]}(s)\sigma^{(2)}_s)\, d\tilde B $,
then the r.h.s.\ above is the cost of $M$ as a martingale starting at $\mu$ and ending at $\nu$, and so is smaller than $V(0,1,\mu,\nu)$.

Let $M^*$ be optimal for $V(0,1,\mu,\nu)$. Using \eqref{eq DPP} it is trivial to show Points $(1)$-$(2)$. But from this follows that $M^*|_{[0,\tau]}$ is optimal for the r.h.s.\ of \eqref{eq DPP}. This, Point $(1)$, and the arguments in the previous paragraph show how to stitch together $M^*|_{[0,\tau]}$ and $M^*|_{[\tau,1]}$ to produce an optimizer $M$ for $V(0,1,\mu,\nu)$. But this must then coincide with $M^*$, by uniqueness. On the other hand $M_1$ is defined via $M^*_\tau$ and a Brownian motion independent of $\{M^*_s:s\leq \tau\}$, so $\law(M_1|M^*_s,s\leq\tau)= \law(M_1|M^*_\tau)$ and we conclude.
\end{proof}

{

\begin{proposition}
\label{disp mart interpolation}
Let $M^*$ be the optimizer of \eqref{MBMBB} and set $$[\mu,\nu]^M_t:=\law(M^*_t).$$
Then $\law(M^*_0,M^*_t)$ is optimal for \eqref{static weak transport} between the marginals $\mu$ and $[\mu,\nu]^M_t$. Similarly, the optimizer of \eqref{MBMBB} between the same marginals is the time-changed martingale $s\in[0,1]\mapsto M^*_{st}$. Finally, for $0\leq r\leq t\leq 1$, we have
\begin{align}\label{eq disp mart interpolation}
WT(\,[\mu,\nu]^M_r\,,\,[\mu,\nu]^M_t)= MT(\,[\mu,\nu]^M_r\,,\,[\mu,\nu]^M_t)=\sqrt{t-r}\,MT(\mu,\nu) =\sqrt{t-r}\,WT(\mu,\nu). 
\end{align}
\end{proposition}

\begin{proof}
We use the notation in Remark \ref{rem connection} and write $M^*_t=M^X_t=f^X_t(B_t)$ where $$\textstyle f^x_t(\cdot):= \int \nabla F^x(b+\cdot)\gamma^d_{1-t}(db).$$ Since $[\mu,\nu]^M_t=f_t^X(\sqrt{t}B_1)$, it is not difficult to see that $$N^*_s:=\E[f_t^X(\sqrt{t}B_1)|{\cal F}^B_s]=f_{st}^X(\sqrt{t}B_s),$$
is the optimizer of \eqref{MBMBB} from $\mu$ at $s=0$ to $[\mu,\nu]^M_t$ at $s=1$. Of course $N^*$ coincides (in law) with the time-changed martingale $s\mapsto M^*_{st}$, and by Theorem \ref{lem inequality static dynamic} we get the optimality of $\law(M^*_0,M^*_t)$.  We next remark that $J(f^X_s)(B_s)$ is a matrix-valued martingale, where $J$ stands for Jacobian, as can be easily seen from the convolution structure or PDE arguments. 
Thus $\E[J(f_s^X)(B_s)]=\E[J(f_{st}^X)(\sqrt{t}B_s)]$. To recognize the ``$\sigma$'' of $N^*$ and $M^*$ we observe that 
\begin{align*}
dN^*_s &= \sqrt{t}J(f^X_{st})(\sqrt{t}B_s)dB_s,\\
dM^*_s & = J(f^X_s)(B_s)dB_s,
\end{align*}
by It\^o formula. Putting all together we find
\begin{align*}
\textstyle \E\left [\int_0^1 \sqrt{t}J(f^X_{st})(\sqrt{t}B_s)ds  \right ]&=\textstyle \sqrt{t} \int_0^1 \E\left [ J(f^X_{st})(\sqrt{t}B_s)\right ]ds  \\
& = \textstyle  \sqrt{t} \int_0^1 \E\left [ J(f^X_{s})(B_s)\right ]ds 
\\
& = \textstyle  \sqrt{t}  \E\left [ \int_0^1J(f^X_{s})(B_s)ds \right ] ,
\end{align*}
and again by Theorem \ref{lem inequality static dynamic} we get $$ MT([\mu,\nu]^M_0\,,\,[\mu,\nu]^M_t)=\sqrt{t}\,MT(\mu\,,\,\nu) .$$
The general case of \eqref{eq disp mart interpolation} follows similarly.
\end{proof}

Since Proposition \ref{disp mart interpolation} shows that $\law(M^*_0,M^*_t)$ is optimal for \eqref{static weak transport} between its marginals, Point $(3)$ of Lemma \ref{lem DPP} immediately implies

\begin{corollary}\label{strong Markov}
The unique optimizer $M^*$ of \eqref{MBMBB} has the strong Markov property.
\end{corollary}

\begin{remark}
The identities \eqref{eq disp mart interpolation}, at least for the continuous-time problems, have been obtained in \cite[Remark 4.1]{HuTr17} in a more general setting, via a scaling argument. The interpretation of \eqref{eq disp mart interpolation} is clear: Our optimal martingale is a constant-speed geodesic when distance is measured wrt the square of our cost functional.
\end{remark}
}

\section{Main results in dimensions one and two}
\label{sec main dim 1 and 2}
In this section we study finer structural properties of the unique optimizer of \eqref{MBMBB} established in the previous section.
We get a full description in dimension one, in dimension two under the additional Assumption \ref{ass negligible boundary} and a partial description in general dimensions.

\subsection{The one-dimensional case}
\label{sec main dim 1}

Let $\mu\leqc \nu$ be probability measures on the line with finite second moment. For a measure $\alpha$ on $\R$ and $x\in \R$ we write $u_\alpha(x):=\int |x-y|\, d\alpha(y)$. The convex order relation $\mu\leqc \nu $ is equivalent to $u_\mu\leq u_\nu$.  

We recall from \cite[Appendix A.1]{BeJu16} that  the ``irreducible components of $(\mu,\nu)$'' are determined by the (unique)  family of open disjoint intervals $\{I_k\}_{k\in\N}$ whose union equals the open set 
$$\textstyle \{u_\mu<u_\nu\}:=\left\{ x\in\R: \int |x-y|\,d(\mu-\nu)(y)\neq 0\right\}.$$ 
One can then decompose $$\textstyle\mu=\eta+\sum_k \mu_k\,\,\,\,\,\,\mbox{ and }\,\,\,\,\,\,\nu=\eta+\sum_k \nu_k,$$ where $\mu_k=\mu|_{I_k}$, with $I_k=\{u_{\mu_k}<u_{\nu_k}\} $ and $\nu_k(\overline{I_k})=\mu_k(I_k)$, whereas $\eta$ is concentrated on $\R\backslash \bigcup_k I_k$. A useful straightforward result is that every martingale coupling from $\mu$ to $\nu$ (i.e.\ $\pi\in \M(\mu,\nu)$) is fully characterized by how it looks on the sets $I_k\times \overline{I_k}$. The restrictions $\pi_k:= \pi_{|I_k\times \overline{I_k}}= \pi_{|I_k\times \overline{\R}}$ are still martingale couplings (in the sense that their respective disintegrations satisfy $\int y\, (\pi_k)_x(y)=x$ for $\mu_k$-a.a.\ $x$) but with total mass $\mu_k(I_k)$ and  marginals $\mu_k, \nu_k$. 

We can now state our main result for $d=1$, characterizing the structure of stretched Brownian motion.

\begin{theorem}\label{MainTheoremOneDim} 
Let $\mu\leqc \nu$ be probability measures on the line with finite second moment. A candidate martingale $M$ is an optimizer of \eqref{MBMBB} if and only if it is a \emph{standard stretched Brownian motion} {on each irreducible component $(\mu_k, \nu_k)$  of $(\mu,\nu)$}.  In particular, \emph{stretched Brownian motion} (\sbm) is a \emph{standard stretched Brownian motion} (\ssbm) in each irreducible component.
\end{theorem}

Let us explain the terminology used here. Saying that $M$ is \ssbm\ on the irreducible components of $(\mu,\nu$)  concretely means that, conditionally on $M_0\in I_k$,  $M$ is a \ssbm\ from $\frac{1}{\mu_k(I_k)}\mu_k$ to $\frac{1}{\mu_k(I_k)}\nu_k$. We stress that in the present $1-d$ case Theorem \ref{MainTheoremOneDim} is significantly stronger than Theorem \ref{ThmStretchedtoStandardIntro}.

We now prove the fact, first mentioned in the introduction, that in $1-d$ the transition kernel of stretched Brownian motion is Lipschitz:

\begin{corollary}\label{coro LM}
Let $\mu\leqc \nu$ be probability measures on the line with finite second moment, and $M^*$ the unique stretched Brownian motion from $\mu$ to $\nu$. Then the kernel
$$x\mapsto \pi_x^*:=\law(M^*_1|M^*_0=x),$$
has the Lipschitz property: $\mathcal W_1(\pi^*_x,\pi^*_{x'})\leq |x-x'|$.
\end{corollary}

\begin{proof}
By Theorem \ref{MainTheoremOneDim}, $M^*$ is a (\ssbm) in each irreducible component. Assume first that $x<x'$ and that they belong to the same component. Conditioning to starting in this component, we can write $M^*_t=\EE[f(B_1)|B_t]$, with $f$ increasing and $B$ a Brownian motion with some starting law. Choose $y,y'$ such that
$$\EE^y[f(B_1)]=x< x' = \EE^{{ y'}}[f(B_1 )]$$
and observe that this implies $y<y'$ since $f$ is increasing. This in turn implies that for a Brownian motion $B^0$ starting in zero
the random vector $$(\,f(B_1^0+y)\,,\,f(B_1^{0} +y')\,),$$
is ordered and has marginals $\pi^*_x$ and $\pi^*_{x'}$. Hence
\begin{align*}
\mathcal{W}_1(\pi^*_x,\pi^*_{ x'})&\leq \EE[\,|\, f(B_1^{0}+ y')- f(B_1^0+y)  \,|\,]= \EE^{ y'}[f(B_1)]-\EE^y[ f(B_1)  ]= x' - x.
\end{align*}
On the other hand, if $x,x'$ are not in the same component, we let $f$ and $g$ denote the increasing functions associated to the representations in terms of (\ssbm)'s. If $x<x'$ then the range of $f$ lies below the range of $g$, and we conclude much as in the above display.
\end{proof}

We now proceed towards the subtler extension of Theorem \ref{MainTheoremOneDim} for $d=2$. We omit the proof of Theorem \ref{MainTheoremOneDim} since it is easily derived from the two-dimensional considerations (with less effort and without the additional assumptions). 

\subsection{Preliminaries}\label{sec preliminaries}

We briefly discuss some of the aspects related to the decomposition of martingale couplings in arbitrary dimensions. Later this will be mostly used in dimension two. After this, we also provide an analytical result of much importance for the next sections.

\begin{definition}
A convex paving $\CCC$ is a collection of disjoint relatively open convex sets from $\R^d$. Denoting $\bigcup \CCC := \bigcup_{C\in\CCC}C$, we will always assume $\mu(\bigcup \CCC)=1$ for such objects. For $x\in \bigcup\CCC\subset\R^d$ we denote by $C(x)$ the unique element of $\CCC$ which contains $x$. We say that $\CCC$ is measurable (resp.\ $\mu$-measurable, universally measurable) if the function $x\mapsto \overline{C(x)}$ is measurable as a map from $\R^d$ to the Polish space of all closed (convex) subsets of $\R^d$ equipped with the  \textit{Wijsman} topology\footnote{The Wijsman topology on the collection of all closed subsets
of a metric space ($\mathsf{X},d)$ is the weak topology generated by $\{ \mbox{dist}(x,\cdot) : x\in\mathsf{X}\} $, cf.\ \cite{DMTo17}.}. 
\end{definition}

\begin{definition}
Let $\Gamma\subset \R^d\times \R^d$ and $\pi\in \M (\mu,\nu) $. We say that a convex paving $\CCC$ is
\begin{itemize}
\item $\pi$-invariant if $\pi_x(\, \overline{C(x)} \,)=1$ for $\mu$-a.e.\ $x$,
\item $\Gamma$-invariant if $\mbox{ri}(\Gamma_x)\subset C(x)$ for all $x\in \mbox{proj}_1(\Gamma)$.
\end{itemize}
\end{definition}

Note that a natural order between convex pavings $\CCC,\CCC'$ is given by
$$\CCC\leq_\mu\CCC'\,\,\iff\,\, C(x)\subset C'(x)\,\, \mbox{ for }\mu-a.e.\, x,$$
in which case we say that $\CCC$ is finer than $\CCC'$  (and the latter is coarser than the former). The following two theorems are shown in \cite{GhKiLi16b,DMTo17,ObSi17}.

\begin{theorem}[Ghoussoub-Kim-Lim \cite{GhKiLi16b}]
Given $\pi\in \M (\mu,\nu)$ and $\Gamma\subset \R^d\times \R^d$ a martingale support for $\pi$, there is a finest $\Gamma$-invariant convex paving. We denote it by $\CCC_{\pi,\Gamma}$.
\end{theorem}

\begin{theorem}[De March-Touzi \cite{DMTo17}, Ob{\l}{\'o}j-Siorpaes \cite{ObSi17}] \label{thm DmTOS}There is a finest  convex paving, denoted $\CCC_{\mu,\nu}$, which is $\pi$-invariant for all $\pi\in \M (\mu,\nu)$ simultaneously. Writing $\CCC_{\mu,\nu}=\{C_{\mu,\nu}(x)\}_{x\in \R^d}$, the function $x\mapsto C_{\mu,\nu}(x)$ is universally measurable.
\end{theorem}

If we knew that these convex pavings coincide, this would streamline some of our proofs. For the case $d=1$ this is indeed the case, but already for $d=2$ this can fail. We will actually use another convex paving which incorporates ideas/properties from the above two. 

\begin{lemma}\label{lem existence paving}
Given $\pi\in \M (\mu,\nu)$ there is a finest measurable $\pi$-invariant convex paving, which we denote $\CCC_\pi$.
\end{lemma}

This can be established by a close reading of \cite{DMTo17}, and adapting the arguments therein (of course \cite{DMTo17} achieves much more!). We give a self-contained, shorter argument 
%in Section \ref{subsec converse} below 
under the following  additional hypothesis, which will also appear in Section \ref{sec main dim 2}.

\begin{assumption}\label{ass negligible boundary}For all $\pi\in \M (\mu,\nu)$ and $\CCC \mbox{ convex paving}$ we have $$ \pi_x(\,\overline{C(x)}\,)=1 \, \mu-a.s.\,\,\Rightarrow \pi_x(C(x))=1\,\mu-a.s.$$
In particular, for such $\CCC$ and $\pi$, $\CCC$ is $\pi$-invariant iff $\pi_x(C(x))=1\,\mu-$a.s.
\end{assumption}

\begin{proof}[Proof of Lemmma \ref{lem existence paving} under Assumption \ref{ass negligible boundary}]
Inspired by \cite{DMTo17}, we introduce the optimization problem 
$$\textstyle \inf \{ \int\mu(dx)\,G(C(x)):\,\CCC \mbox{ is a $\pi$-invariant measurable convex paving} \},$$
where $G(C):= dim(C)+g_C(C)$ and $g_C$ is the standard Gaussian measure on $\mbox{aff}(C)$, i.e.\ as obtained from the  $\mbox{dim}(C)$-dimensional Lebesgue measure on $\mbox{aff}(C)$. Let $\CCC^n$ be an optimizing sequence of $\pi$-invariant convex pavings and let $\Omega$ be a set of $\mu$-full measure on which we have $\pi_x(\,\overline{C^n(x)}\,)=1$ for all $n$ (here $C^n(x)$ denotes an element of $\CCC^n$). Introduce for $x\in \Omega$ the relatively open convex sets $C_\pi(x):=\text{rf}_x\left( \bigcap C^n(x) \right )$. We have\footnote{Recall that $A\subset A '\Rightarrow \text{rf}_a\, A \subset \text{rf}_a\, A'$, that $a\in A\iff a\in\text{rf}_a(A)$ and that $\text{rf}_a(A)=\text{ri}\,A\iff a\in \text{ri}\,A $. } $x\in C_\pi(x)$ since $x\in \bigcap C^n(x) $. Moreover we have that $\CCC_\pi:=\{C_\pi(x):x\in\Omega\}$ forms a partition since already  $\{\bigcap C^n(x):x\in\Omega\}$ is a partition. Let us establish that $\pi_x(\overline{C_\pi(x)})=1$.

We start assuming
\begin{align}
\label{eq claim 1}
\forall \,K\,\,\text{convex}:\,\,\text{ri}\,\overline{\text{co}}\,\text{supp}\,\pi_x\subset K\Rightarrow \pi_x\left(\overline{\text{rf}_xK}\right)=1.
\end{align}
Let us take $K:=\bigcap C^n(x)$. Since $\overline{C^n(x)}$ is closed, convex and satisfies $\pi_x(\overline{C^n(x)})=1$ we have $\overline{\text{co}}\,\text{supp}\,\pi_x\subset \overline{C^n(x)}$. On the other hand, $\overline{\text{co}}\,\text{supp}\,\pi_x$ cannot be contained in $\partial C^n(x)$ since by Assumption  \ref{ass negligible boundary} we have $\pi_x(\partial C^n(x))=0$. By \cite[Corollary 6.5.2]{Ro70} we must then have $\text{ri}\,\overline{\text{co}}\,\text{supp}\,\pi_x\subset\text{ri}\, C^n(x)=C^n(x)$ for all $n$, so $\text{ri}\,\overline{\text{co}}\,\text{supp}\,\pi_x\subset\bigcap C^n(x)=K$. By \eqref{eq claim 1} we get $ \pi_x\left(\overline{\text{rf}_xK}\right)=  \pi_x\left(\overline{C_\pi(x)}\right)=1$ as desired. All in all $\CCC_\pi$ is a $\pi$-invariant convex paving, and since $C_\pi(x)\subset C^n(x)$ we find $\int\mu(dx)G(C_\pi(x))\leq \int\mu(dx)G(C^n(x))$ from which we get the optimality of $\CCC_\pi$.

To finish the proof, let us establish \eqref{eq claim 1}. By the martingale property we easily see\footnote{Let $m=\text{dim}(\overline{\text{co}}\,\text{supp}\,\pi_x)$ and suppose $x\in\partial( \overline{\text{co}}\,\text{supp}\,\pi_x)$. We can then find an $(m-1)$-dimensional hyperplane supporting $x$ and having $\overline{\text{co}}\,\text{supp}\,\pi_x$ contained in one associated half-space. By the martingale property one obtains that necessarily $\text{supp}\,\pi_x$, and then $\overline{\text{co}}\,\text{supp}\,\pi_x$ too, must be actually contained in the hyperplane itself. Thus $\text{dim}(\overline{\text{co}}\,\text{supp}\,\pi_x) \leq m-1$ yielding a contradiction.} that $x\in \text{ri}\,\overline{\text{co}}\,\text{supp}\,\pi_x$. From this, $\text{ri}\,\overline{\text{co}}\,\text{supp}\,\pi_x =\text{rf}_x\,\left( \text{ri}\,\overline{\text{co}}\,\text{supp}\,\pi_x\right )\subset \text{rf}_x\,K$. Hence $\overline{\text{ri}\,\overline{\text{co}}\,\text{supp}\,\pi_x} \subset \overline{ \text{rf}_x\,K}$, whose l.h.s.\ equals $\overline{\text{co}}\,\text{supp}\,\pi_x$ by \cite[Theorem 6.3]{Ro70}, so \eqref{eq claim 1} follows.
\end{proof}

\begin{remark}
The same proof, modulo obvious changes, proves the existence of a finest measurable convex paving invariant for all $\pi\in \M (\mu,\nu)$ simultaneously. This however does not establish the existence of a maximally spreading martingale coupling as in \cite{DMTo17}.
\end{remark}

Here is a  sufficient criterion for Assumption \ref{ass negligible boundary} to hold.

\begin{lemma}\label{lem sufficiency for negligible boundary}
Assumption \ref{ass negligible boundary} is satisfied if
$d\in \{ 1,2\}$ and $\nu\ll\lambda^d$.
\end{lemma}

\begin{proof}
This follows by similar arguments as in \cite[Lemma C.1]{GhKiLi16b}. We omit the details.
\end{proof}

A direct consequence of Theorem \ref{thm DmTOS} and Assumption \ref{ass negligible boundary} is the decomposition of a martingale into irreducible components. Notice the resemblance to the one-dimensional case explained in Section \ref{sec main dim 1}.

\begin{proposition}\label{prop martingale decomposition}
Let $\CCC_{\mu,\nu}=\{C_{\mu,\nu}(x)\}_{x\in\R^d}$ be the convex paving of Theorem \ref{thm DmTOS} and assume Assumption \ref{ass negligible boundary}. Then
\begin{itemize}
\item[(i)] we may decompose  $$\textstyle \mu=\int \mu(~\cdot~ |K) dC_{\mu,\nu}(\mu)(K), \mbox{ and }\, \nu = \int \nu(~\cdot~ |K) dC_{\mu,\nu}(\mu)(K), $$ with $\mu(~\cdot~|K)\leqc \nu(~\cdot~|K)$ for $C_{\mu,\nu}(\mu)$-a.e.\ $K$;
\item[(ii)] for any martingale coupling $\pi\in \M (\mu,\nu)$ we have that $$\pi(~\cdot~|K\times K) = \pi(~\cdot~|K\times \R^d\,)\, \mbox{ for } \, C_{\mu,\nu}(\mu)-a.e.\, K, $$
and this common measure has first and second marginals equal to $\mu(~\cdot~| K)$ and $\nu(~\cdot~| K)$ respectively;
\item[(iii)] any martingale coupling $\pi\in \M (\mu,\nu)$ can be uniquely decomposed as $$\textstyle \pi = \int \pi(~\cdot~|K\times K)dC_{\mu,\nu}(\mu)(K).$$
\end{itemize}
\end{proposition}
\noindent
The proof is just as in \cite[Appendix A.1]{BeJu16}, but simpler, thanks to the fact that under Assumption \ref{ass negligible boundary} we have that \emph{martingales started on two neighbouring cells will not go on to reach the intersection of the boundaries of the cells}. We thus omit the proof.

We finally present a technical lemma which will be extremely useful in the proofs of the main results in dimension two.

\begin{lemma}\label{lem technical bla}
Let $\eta$ be a probability measure in $\mathbb{R}^d$ with finite second moment, and ${F}:\mathbb{R}^d\to \mathbb{R}$ convex such that $\nabla {F}(\gamma^d)=\eta$. Denote $V:=\text{aff}(\text{supp}(\eta))$ and let $P$ be the orthogonal projection onto $V$.  Then, there exists a convex function $\tilde F:V\to\mathbb{R}$  such that $\gamma^d-a.s.\,\,\,\nabla {F} = \nabla  \tilde F\circ P$.
For all $s>0$, the function
$$\textstyle \mathbb{R}^d\ni b\mapsto f_s(b):=\int \nabla F(b+y)\gamma^d_{s}(dy) = \int \nabla \tilde F( Pb+z)P(\gamma^d_{s})(dz) \in \mathbb{R}^d,$$  
has the following properties:
\begin{enumerate}
\item It is infinitely continuously differentiable.
\item Restricted to V, it is one-to-one.
\item  $\overline{f_s(\mathbb{R}^d)}=\overline{\text{co}}\,\nabla F(\mathbb{R}^d)$.
\item  $f_s(\gamma^d)$ is equivalent to the $m$-dimensional Lebesgue measure on $V$ restricted to ${\text{co}}\,\nabla F(\mathbb{R}^d)$, where $m=\text{dim}(V)$.%(\overline{\text{co}}\,\nabla F(\mathbb{R}^d))$.
\item $\text{supp}(f_s(\gamma^d_t))=\overline{\text{co}}\,\nabla F(\mathbb{R}^d)$ is convex and does not depend on $s>0$ nor $t>0$.
\end{enumerate}
\end{lemma}

\begin{proof}
The $\gamma^d-a.s.$ equality $\nabla {F} = \nabla  \tilde F\circ P$, follows from Brenier's Theorem by taking $\nabla\tilde F$ mapping $P(\gamma^d)$ into $\eta$ and observing that $\nabla (\tilde F\circ P)=P((\nabla \tilde F )\circ P)=(\nabla \tilde F )\circ P$. Point (1) follows by change of variables and differentiation under the integral sign. Alternatively, one can argue with the classical backwards heat equation. Points (2), (3) and (5) follow by the full-support property of $\gamma^d$ in $\mathbb{R}^d$ and $P(\gamma^d)$ in $V$. 

Point 4 is trivially true if $\eta$ is a Dirac delta (then $m$=0). Otherwise it suffices to consider the smooth function $V\ni v\mapsto\tilde{f}_s(v):=  \int \nabla \tilde F( v+z)\tilde \gamma(dz)$, with $\tilde \gamma =P(\gamma^d_{s})$, and to prove that  $\tilde{f}_s(\tilde \gamma) \sim \lambda_V|_{{\text{co}}\,\nabla F(\mathbb{R}^d)}$, where the latter denotes $m$-dimensional Lebesgue on $V$ restricted to ${\text{co}}\,\nabla F(\mathbb{R}^d)$.  Since $\tilde \gamma\sim \lambda_V$, we have by \cite[Theorem 4.8(i)]{Vi03} that $\lambda_V$-a.e.\ the Jacobian of $\tilde{f}_s$ is invertible. By the change of variables formula, it is easy to obtain that $\tilde f_s(\tilde \gamma)\ll \lambda_V$, and the previous observation with the Monge-Amp\`ere equation \cite[Theorem 4.8(iii)]{Vi03} yield
\begin{align}
\lambda_V-a.e.\ r:\,\,\,\,
\frac{d \tilde f_s(\tilde \gamma)}{d \lambda_V}(r) & = \left |\text{det}\left((J\tilde f_s)^{-1}(r)\right )\right | \frac{d\tilde \gamma}{d \lambda_V}\left( (\tilde f_s)^{-1}(r)\right) {\bf 1}_{\tilde f_s(V)}(r).
\end{align}
By Point 3, $\tilde f_s(V)={\text{co}}\,\nabla F(\mathbb{R}^d) $, and so we conclude $\tilde f_s(\tilde \gamma)\sim \lambda_V|_{{\text{co}}\,\nabla F(\mathbb{R}^d)}$ since under the latter measure the density $\frac{d \tilde f_s(\tilde \gamma)}{d \lambda_V|_{{\text{co}}\,\nabla F(\mathbb{R}^d)}}$ is a.e.\ non-vanishing. 
\end{proof}

\subsection{The two-dimensional case}
\label{sec main dim 2}

\comment{JB: It seemed to me that we needed to explain why we put so much effort in this Thm. Thus I added this paragraph.}
Our first main result for $d=2$ is a characterization of the structure of \sbm, providing a significantly  strengthened version of Theorem \ref{ThmStretchedtoStandardIntro} in the introduction. \comment{MB: I agree that it would be good to have something like this. However I couldn't grasp the explanation why if it is the deepest result of the article.}

\begin{theorem}\label{thm main} Let $\mu\leqc\nu$ be probability measures in $\mathbb{R}^2$ with finite second moments. Suppose $\nu\ll\lambda^2$, and let $M^*$ be the unique optimizer for \eqref{MBMBB}. Set $\pi^t=\law(M^*_0,M^*_t)$ for $0<t<1$. Then  the stretched Brownian motion $M^*$ is  a standard stretched Brownian motion on each cell of $\CCC_{\pi^t}$.
\end{theorem}

The second main result of this part is the optimality of \ssbm\ whenever we are able to build them with respect to the coarser $\CCC_{\mu,\nu}$ convex paving. Our proof of such result relies on the  simplifying Assumption \ref{ass negligible boundary}, which as seen in Lemma \ref{lem sufficiency for negligible boundary} is \emph{verified in dimension two under the further requirement that $\nu$ be absolutely continuous}. We therefore place this result here, although in principle it is a result valid in arbitrary dimensions.

\comment{JB: Parts of what used to be here were redundant with Theorem \ref{ThmStandardtoStretchedIntro}, so I chopped them away.}
\begin{theorem}\label{thm converse}
Under Assumption \ref{ass negligible boundary}, if $M$ is a standard stretched Brownian motion on each cell\footnote{This means that for $C_{\mu,\nu}(\mu)$-a.e.\ $K$, the conditioning of $M^*$ to $M_0^*\in K$ is a stretched Brownian motion between the marginals $\mu(\cdot\,|K)$ and $\nu(\cdot\,|K)$  introduced in Proposition \ref{prop martingale decomposition}} of the convex paving $\CCC_{\mu,\nu}$, then it is optimal for \eqref{MBMBB} \comment{MH: I think we should stress more for which results we really need the assumption \ref{ass negligible boundary}} (i.e. it is a \sbm).
\end{theorem}

\begin{remark}
The difference between Theorem \ref{ThmStandardtoStretchedIntro} and Theorem \ref{thm converse} is as follows: the first result  says that standard stretched Brownian motion is optimal in its own, whereas the second statement allows for more freedom in that we are allowed to choose the convex function in the definition of stretched Brownian motion dependent on the cells of $\CCC_{\mu,\nu}$. Therefore this result is a strengthened version of Theorem \ref{ThmStandardtoStretchedIntro}.
\end{remark}

\begin{remark} 
For dimension one ($d=1$), Theorem \ref{MainTheoremOneDim} establishes the existence of standard stretched Brownian motion, and characterize it as the sole optimizer. Both existence and optimality are understood with respect to the same (countable) convex paving. For two dimensions ($d=2$), Theorems \ref{thm main} and \ref{thm converse} and Lemma \ref{lem sufficiency for negligible boundary} establish, under the assumption that $\nu\ll\lambda^2$, the existence and optimality characterization of standard stretched Brownian motion. In this case however, existence and optimality are understood with respect to potentially different convex pavings.
\end{remark}

The proofs of these results are deferred to Section \ref{sec pending proofs}. Theorem \ref{thm main} relies crucially on a monotonicity principle which we now establish and which seems of independent interest.

{
\section{A monotonicity principle for weak optimal transport problems}\label{sec mono}
For this part only, we adopt a more general setting. Let $\mathsf X,\mathsf Y$ be Polish spaces and $C:\mathsf X\times {\cal P}(\mathsf Y)\to \mathbb{R}\cup\{+\infty\}$ Borel measurable. Consider for $\mu\in {\cal P}(\mathsf X),\nu\in {\cal P}(\mathsf Y)$ the optimization problem
\begin{align}\label{eq defi gen Gozlan}
\inf_{\pi\in\Pi(\mu,\nu)}\int_X\mu(dx)C(x,\pi_x). 
\end{align}
\noindent This is a weak (i.e.\ non-linear) transport problem in the sense of \cite{GoRoSaTe14,GoRoSaTe15} and the references therein. We now obtain a ``monotonicity principle'' for this problem, i.e.\ a finitistic ``zeroth-order'' necessary optimality condition.

\begin{proposition}
\label{prop monotonicity general}
Suppose that
\begin{itemize}
\item Problem \eqref{eq defi gen Gozlan} is finite with optimizer $\pi$;
\item $C$ is jointly measurable;
\item $\mu(dx)$-a.e. the function $C(x,\cdot)$ is convex and lower semicontinuous.
\end{itemize} 
Then there exists a Borel set $\Gamma\subset X$ with $\mu(\Gamma)=1$ and the following property 
\begin{center} if $x,x'\in \Gamma$ and $m_x,m_{x'}\in {\cal P}(Y)$ satisfy $m_x+m_{x'}=\pi_x+\pi_{x'}$, then\end{center}
$$C(x,\pi_x)+C(x',\pi_{x'})\leq C(x,m_x)+C(x',m_{x'}).$$
\end{proposition}

\begin{proof}
Let 
$$
{\cal D}:=\left\{
\left (\,(x,x'),( m_1,m_2)\,\right )\in X^2\times {\cal P}(Y)^2:\, \begin{array}{l}
 m_1+m_{2}=\pi_x+\pi_{x'}\,\text{, and }\\C(x,\pi_x)+C(x',\pi_{x'})>C(x,m_1)+C(x',m_{2})
\end{array}
\right\},$$
which is an analytic set. By the Jankov-von Neumann uniformization theorem there is \cite[Theorem 18.1]{ke95} an analytically measurable function $$D:=\text{proj}_{X^2}({\cal D})\ni (x,x') \mapsto (m^{(x,x')}_1,m^{(x,x')}_{2})\in {\cal P}(Y)^2,$$ so that $(x,x',m^{(x,x')}_1,m^{(x,x')}_{2})\in {\cal D}$. Since $(x,x',m_1,m_2)\in {\cal D}\iff (x',x,m_2,m_1)\in {\cal D}$, it is possible to prove that we may actually assume that 
\begin{align}\label{eq sym of m}
(m^{(x',x)}_1,m^{(x',x)}_{2})=(m^{(x,x')}_2,m^{(x,x')}_{1}).
\end{align}
Of course the set $D$ is likewise analytic. Thus extending $(m^{(\cdot,\cdot)}_1,m^{(\cdot,\cdot)}_2)$ to $(x,x')\notin D$ by setting it to $(\pi_x,\pi_{x'})$, analytic-measurability and the symmetry property \eqref{eq sym of m} are preserved. 

Assume that there exists $Q\in\Pi(\mu,\mu)$ such that $Q(D)>0$. We now show that this is in conflict with the optimality of $\pi$.  By considering $\frac{Q+e(Q)}{2}$, where $e(x,x'):=(x',x)$, we may assume that $Q$ is symmetric. We first define
\begin{align}\label{eq pi tilde better}\textstyle
\tilde{\pi}(dx,dy):=\mu(dx)\int_{x'}Q_x(dx')m^{(x,x')}_1(dy),
\end{align}
which is legitimate owing to the measurability precautions we have taken. We will prove
\begin{enumerate}
\item $\tilde{\pi}\in \Pi(\mu,\nu)$,
\item $\int\mu(dx) C(x,\pi_x)>\int\mu(dx) C(x,\tilde{\pi}_x)$.
\end{enumerate}
For $(1)$: Evidently the first marginal of $\tilde{\pi}$ is $\mu$. On the other hand
$$\textstyle \int_x\mu(dx)\tilde{\pi}_x(dy)=  \int_x\mu(dx) \int_{x'}Q_x(dx')m^{(x,x')}_1(dy) = \int_{x,x'}Q(dx,dx')m^{(x,x')}_1(dy). $$
The last quantity is equal to $ \int_{x,x'}Q(dx,dx')m^{(x,x')}_{2}(dy) $ by symmetry of $Q$ and \eqref{eq sym of m}. So 
$$\textstyle \int_x\mu(dx)\tilde{\pi}_x(dy)= \int_{x,x'}Q(dx,dx')\frac{m^{(x,x')}_{1}+m^{(x,x')}_{2}}{2}(dy)=\int_{x,x'}Q(dx,dx')\frac{\pi_{x'}+\pi_{x}}{2}(dy)=\nu(dy),$$
by definition of $m_i^{(x,x')}$ and $Q$. Thus $\tilde{\pi}$ has second marginal $\nu$. 

For $(2)$: By convexity of $C(x,\cdot)$, the symmetry of $Q$ and \eqref{eq sym of m}, and by the assumption that on the $Q$-non negligible set $D$ we have $C(x,\pi_x)+C(x',\pi_{x'})>C(x,m^{(x,x')}_1)+C(x',m^{(x,x')}_{2})$, we obtain
\begin{align*}
\textstyle\int_x\mu(dx)C(x,\tilde{\pi}_x)&= \textstyle\int_x\mu(dx)C\left(x,\int_{x'}Q_x(dx')m^{(x,x')}_1\right)\\ &\leq \textstyle \int_x\mu(dx)\int_{x'}Q_x(dx') C\left(x,m^{(x,x')}_1\right)\\ & =\textstyle
\int_{x,x'}Q(dx,dx')C\left(x,m^{(x,x')}_1\right)
\\&\textstyle =\int_{x,x'}Q(dx,dx')  \frac {C\left(x,m^{(x,x')}_1\right)+ C\left(x,m^{(x,x')}_2\right) }{2} \\ & \textstyle <
\int_{x,x'}Q(dx,dx')  \frac {C(x,\pi_x)+ C(x',\pi_{x'})}{2}
\\&\textstyle = \int_x\mu(dx)C(x,\pi_x).
\end{align*}
\noindent As expected, we have contradicted the optimality of $\pi$.

We conclude that no measure $Q$ with the stated properties exists. By ``Kellerer's lemma'' \cite[Proposition 2.1]{BeGoMaSc08}, which is also true for analytic sets, we obtain that $D$ is contained in a set of the form $N\times N$ where $\mu(N)=0$. Letting $\Gamma:= N^c$, so $\Gamma\times\Gamma\subset D^c$, we easily conclude.
\end{proof}

We now go back to the main framework in this article. The monotonicity principle will be crucially used, under the following guise, in order to prove the results in Section \ref{sec main dim 2}. For a kernel $\pi_x(dy)$ and $\tilde\mu(d\tilde x)=\frac12(\delta_x(d\tilde x)+\delta_{x'}(d\tilde x))$ we write $\pi_{\tilde x}(dy)\tilde\mu(d\tilde x)=\frac12(\delta_x\pi_x+\delta_{x'}\pi_{x'}).$

\begin{corollary}
\label{prop monotonicity}
Let $\pi$ be optimal for \eqref{static weak transport}. Then there exists $\Gamma\subset \mathbb{R}^d$ with $\mu(\Gamma)=1$ such that
\begin{center}
if $x,x'\in \Gamma$, then the measure $\frac{\delta_x\pi_x+\delta_{x'}\pi_{x'}}{2} $ is optimal for
\end{center}
\begin{align}\label{prob discrete mart}
\inf_{\substack{\text{mean}(m_{x'})=x',\, \text{mean}(m_x)=x \\ (m_x+m_{x'})/2=(\pi_x+\pi_{x'})/2 }} \left\{ \mathcal{W}_2(m_x,\gamma^d)^2 + \mathcal{W}_2(m_{x'},\gamma^d)^2  \right \}.
\end{align}
\end{corollary}

\begin{proof}
Consider Proposition \ref{prop monotonicity general}, taking $X=Y=\mathbb{R}^d$ and setting
$$C(x,m)= \mathcal{W}_2(m,\gamma^d)^2,$$
if $mean(m)=x$ and $+\infty$ otherwise. It is immediate that $C(x,\cdot)$ is convex and lower semicontinuous. Taking $\Gamma$ to be the $\mu$-full set given by Proposition \ref{prop monotonicity general}, the result follows.
\end{proof}

Observe that Problem \eqref{prob discrete mart} is of the same kind as \eqref{static weak transport}, with initial marginal $\frac{\delta_x+\delta_{x'}}{2}$ and terminal marginal $\frac{\pi_x+\pi_{x'}}{2}$. It follows as in Theorems \ref{lem inequality static dynamic} and Lemma \ref{lem DPP} that \eqref{prob discrete mart} has a continuous-time analogue, which enjoys the dynamic programming principle, and whose optimizer is a strong Markov martingale. This fact will be repeatedly used in the next part.

\begin{remark}
 Of course there are versions of the results in this section for general $n$-tuples instead of pairs. Since we only use the version with pairs we did not state the result in its most general form (cf.\ \cite{GoJu18, BaBePa18}).
\end{remark}

}

\section{Pending proofs}
\label{sec pending proofs}

\subsection{Proof of Theorems \ref{ThmStandardtoStretchedIntro} and \ref{thm converse}}\label{subsec converse}

\begin{proof}[Proof of Theorem \ref{ThmStandardtoStretchedIntro}]
Let $A:\R^d\to\R^d$ be in $L^2(\mu)$ and $\phi,\psi:\R^d\to \R$ be conjugate convex functions. We start by proving that 
\begin{align}\textstyle \label{eq aux W lew P}
WT \leq \int \phi \, d\nu -\int x\cdot A(x)\, d\mu +\int\mu(dx)\int \gamma^{(A(x))}(db)\psi (b),
\end{align}
where $\gamma^{(a)}:=\delta_a*\gamma^d$. First observe that $$\textstyle \sup_{q\in\Pi(\pi,\gamma)}\int q(dm,db)m\cdot b = \sup_{q\in\Pi(\pi,\gamma^{(a)})}\int q(dm,db)m\cdot [b-a].$$
Let us write $\textstyle\Sigma := \{ \, \{\pi_x\}_x :\,  \text{mean}(\pi_x)=x \mbox{ and } \int\mu(dx)\pi_x(dy)=\nu(dy)\, \} $. From here,
\begin{align*}
WT &= \sup_{\{\pi_x\}_x \in \Sigma}\int \mu(dx) \sup_{q\in\Pi(\pi_x,\gamma^{(A(x))})}\int q(dm,db)m\cdot [b-A(x)] \\
 &= \sup_{\{\pi_x\}_x \in \Sigma}\int \mu(dx)\left [-\int\pi_x(dm)m\cdot A(x) + \sup_{q\in\Pi(\pi_x,\gamma^{(A(x))})}\int q(dm,db)m\cdot b\right ]\\
  &= \sup_{\{\pi_x\}_x \in \Sigma}\int \mu(dx)\left [-x\cdot A(x) + \sup_{q\in\Pi(\pi_x,\gamma^{(A(x))})}\int q(dm,db)m\cdot b\right ]\\
  &\leq -\int x\cdot A(x)\, d\mu + \sup_{\{\pi_x\}_x \in \Sigma}\int\mu(dx) \sup_{q\in\Pi(\pi_x,\gamma^{(A(x))})}\int q(dm,db)[\phi(m)+\psi(b)]\\
&= -\int x\cdot A(x)\, d\mu + \int \phi~d\nu + \int\mu(dx)\int \gamma^{(A(x))}(db)~\psi(b)
\end{align*}
by the conjugacy relationship $m\cdot b\leq \phi(m)+\psi(b)$ and the defining property of $\Sigma$. Hence, \eqref{eq aux W lew P} follows.

Let now $M$ be standard stretched Brownian motion from $\mu$ to $\nu$ in the notation of Definition \ref{def d dim ssBm} and Equation \eqref{eq f_t}. By classical convex analysis arguments, or optimal transport theory, there exists $\phi,\psi$ convex conjugate functions such that $\lambda^d$-a.e. ($\gamma^d$-a.e.)\
$$\nabla F(b)\cdot b = \phi(\nabla F(b))+ \psi(b).$$
We also choose $A(x)=f_0^{-1}(x)$, which is well defined on $\text{supp}(\mu)$ by Lemma \ref{lem technical bla}.

By definition $\mu\sim M_0=f_0(B_0)\sim f_0(\alpha)$, so  %
$$\textstyle \int x\cdot A(x)\, d\mu(x) =  \int x\cdot A(x)df_0(\alpha)(x)=\int f_0(x)\cdot x \,d\alpha(x) = \mathbb{E}[f_0(B_0)B_0]= \mathbb{E}[M_0\cdot B_0].$$ 
On the other hand 
\begin{align*}\textstyle
\int\phi \, d\nu +\int\mu(dx)\int \gamma^{(A(x))}(db)\psi (b)& \textstyle = \mathbb{E}[\phi(M_1)] + \int A(\mu)(dx)\int \gamma^{(x)}(db)\psi(b)\\
& \textstyle = \mathbb{E}[\phi(M_1)] + \int A(\mu)(dx) \mathbb{E}[\psi(B_1)|B_0= x]\\
& \textstyle = \mathbb{E}[\phi(M_1)] + \mathbb{E}[\,\, \mathbb{E}[\psi(B_1)|B_0= x]\,\,],\\
&=\textstyle \mathbb{E}[\phi(M_1)+ \psi(B_1)],
\end{align*} 
since $A(\mu)=\alpha=\law(B_0)$. So the r.h.s.\ of \eqref{eq aux W lew P} becomes in this case
\begin{multline*}\textstyle\mathbb{E}[\phi(M_1)+ \psi(B_1)-M_0\cdot B_0]= \mathbb{E}[\phi(\nabla F(B_1))+ \psi(B_1)-M_0\cdot B_0]\\ \textstyle = \mathbb{E}[\nabla F(B_1)\cdot B_1-M_0\cdot B_0]=\mathbb{E}[M_1\cdot B_1-M_0\cdot B_0]= \mathbb{E}\left[ \int_0^1 \text{tr}(\sigma_t)dt \right ].
\end{multline*}
Hence, Theorem \ref{lem inequality static dynamic} implies the optimality of $M$.
\end{proof}

We now work under Assumption \ref{ass negligible boundary}, still in arbitrary dimension $d$.

\begin{proof}[Proof of Theorem \ref{thm converse}]
We observe from Proposition \ref{prop martingale decomposition} that the optimization problem \eqref{static weak transport} can be decomposed / disintegrated along the cells of $\CCC_{\mu,\nu}$. Therefore, optimality must only hold for $C_{\mu,\nu}(\mu)$-a.e.\ $K$ for the corresponding transport problems with first and second marginals $\mu(~\cdot ~|K)$ and $\nu(~\cdot~|K)$ respectively. This reduces the argument to the previous case of Theorem \ref{ThmStandardtoStretchedIntro}, and we conclude.
\end{proof}

\subsection{Proof of Theorem \ref{thm main}}

Although this is eventually a two-dimensional result, for the arguments we do not fix the dimension $d$ to two  unless we explicitly say so.

Let $M$ be the unique optimizer of \eqref{MBMBB}, where we drop the superscript $*$ for simplicity. By Theorem \ref{lem inequality static dynamic} this continuous-time martingale is associated to the unique two-step martingale $\pi$ optimizing \eqref{static weak transport}. 
Let $\nabla F^x$ be the optimal transport map pushing $\gamma^d$ to $\pi_x$.

By Remark \ref{rem connection}, we know that conditioning on $M_0=x$ the martingale $M$ is given by
\begin{align}
\textstyle M^x_t: = f^x_t(B_t),\,\,\mbox{ where }\,\, f^x_t(\cdot):= \int \nabla F^x(b+\cdot)\gamma^d_{1-t}(db).\label{eq Mx}
\end{align}

We fix $0<t<1$ throughout. By Lemma \ref{lem technical bla} we find $B_t=(f^x_t)^{-1}(M^x_t)$. We denote
\begin{align}\label{eq pixz}
\pi_{x,y}:=\law(M_1|M_0=x,M_t=y)=  \nabla F^x(\delta_{(f^x_t)^{-1}(y)}*\gamma^d_{1-t}).
\end{align}
{\bf Important convention:} For the rest of this section we make the convention that $x,y,z$ denote possible values of the random variables $M_0,M_t,M_1$ respectively.
\begin{lemma}\label{lem type 1}
Let $g$ be the unique gradient of a convex function such that $g(\gamma_{1-t}^d)=\pi_{x,y}$. Then $\nabla F^x(\cdot)= g(-(f_t^x)^{-1}(y)+\cdot)$. In particular $\nabla F^x$ is uniquely determined by the family of translates of $g$, which we denote by
$$\text{type}(\pi_{x,y}):=\{ a\mapsto g(a-r):\, r\in\mathbb{R}^d\}.$$
\end{lemma}

\begin{proof}
For $r\in\R^d$ write $g_r(\cdot)=g(\cdot -r)$. Then, we have $\pi_{x,y}=g(\gamma_{1-t}^d)=  g_{(f_t^x)^{-1}(y)}(\delta_{(f^x_t)^{-1}(y)}*\gamma^d_{1-t})$. Hence, both $\nabla F^x$ and $g_{(f_t^x)^{-1}(y)}$ push forward $\delta_{(f^x_t)^{-1}(y)}*\gamma^d_{1-t}$ into $\pi_{x,y}$, and both are gradients of convex functions. By the uniqueness result in Brenier's theorem, it follows that they are equal. Thus knowing $\nabla F^x$ determines $g$ modulo translation. Conversely, knowing  $\text{type}(\pi_{x,y})$ (i.e.\ the translations of $g$) determines $\nabla F^x$ upon finding the vector $r$ such that $\int g(r+a)\gamma^d_{1-t}(da)=y$.
\end{proof}

Let 
\begin{align}\label{eq pit}
\pi^{t}:=\law(M_0,M_t)
\end{align}
and consider $$\CCC:=\{\CCC(x)\}_x:=\CCC_{\pi^t}\, ,$$ the minimal $\pi^{t}$-invariant measurable convex paving of Lemma \ref{lem existence paving}. We need to show that on each cell of $\CCC$, $M$ is a standard stretched Brownian motion, i.e.\  on each cell $C(x)$,  we need to find a convex function $F=F_{C(x)}$ such that 
\begin{align}
M^x_1= \nabla F( (f_0)^{-1}(x)+B_1 ),
\label{eq to prove}
\end{align}
where $f_0$ and $F$ are related as in \eqref{eq Mx}.
To this end, we introduce
$$A(x):=\text{type}(\pi_x)=\{ a\mapsto \nabla F^x(a-r)\,:\,\, r\in\mathbb{R}^d\}$$
and we need to show that on each cell $A(x)$ is constant. We start by establishing a few preliminary results.

\begin{lemma}\label{lem A constant}
If $A(x)$ is constant in each cell of $\CCC$, then $M$ is a standard stretched Brownian motion on each of these cells.
\end{lemma}
\begin{proof}
As in Lemma \ref{lem type 1}. Fix arbitrary $ x'\in\CCC(x)$.  Then, we have $\nabla F^{x}(\cdot)= \nabla F^{x'}(r(x)+\cdot)$ as $A(x)=A(x')$. Setting $\nabla F:=\nabla F^{x'}$, proves the claim.
\end{proof}

To make use of the previous lemma, we shall study the behaviour of the martingale $M$ for times in $[0,t]$ and $[t,1]$. Let
\begin{align*}
%\xi(dx,dz)&:= \law(M_0,M_t),\\
\tilde{\pi}(dy,dz)&:= \text{argsup}\,\, V(t,1,\law(M_t),\nu),
\end{align*}
where $\tilde{\pi}$ is understood as the coupling of the initial and terminal marginals of the unique optimizer for $V(t,1,\law(M_t),\nu)$. For $\pi^t$ from \eqref{eq pit} we denote its disintegration w.r.t.\ the second marginal by $(\pi^t_y)_y$. Recall $\pi_{x,y}$ from \eqref{eq pixz}.

\begin{lemma}\label{lem pixz pitilde z}
For $\law(M_t)$-a.e.\ $y$ and $\pi^t_y$-a.e.\ $x$, we have $$\pi_{x,y}(dz)=\tilde{\pi}_y(dz).$$ 
%Equivalently, $$\pi^t(\,\{(x,z): \text{type}(\pi_{x,z}) = \text{type}(\tilde{\pi}_{z})\}\, )=1.$$
\end{lemma}

\begin{proof}
We must have $\tilde{\pi}=\law(M_t,M_1)$, by Lemma \ref{lem DPP} (1). Thus $\tilde{\pi}_y=\law(M_1|M_t=y)$. On the other hand, $\pi_{x,y}=\law(M_1|M_t=y,M_0=x)$ so by Lemma \ref{lem DPP} we get $\pi_{x,y}(dz)=\tilde{\pi}_y(dz)$ for $\law(M_t)$-a.e.\ $y$ and $\pi^t_y$-a.e.\ $x$.
\end{proof}
 
 The previous lemma shows that the type of $\pi_{x,y}$ in fact  only  depends on $y$. Indeed, the same applies also for $x$:

\begin{lemma}\label{lem pixz = A}
For $\mu$-a.e.\ $x$ and $\pi^t_x$-a.e.\ $y$ we have $$\text{type}(\pi_{x,y})=A(x).$$
\end{lemma}

\begin{proof}
By Lemma \ref{lem type 1}, if $g\in \text{type}(\pi_{x,y})$ then $\nabla F^x$ is a translate of $g$ (the translation may depend on $x,y$). But this means conversely that $g$ is a translate of $\nabla F^x$, i.e.\ $g\in A(x)$. Reversing the steps gives the equality.
\end{proof}

We finalize the proof of Theorem \ref{thm main}. In a nutshell, the key is to deal with the null sets in Lemmas \ref{lem pixz pitilde z}-\ref{lem pixz = A}.  \textbf{Only from now on we assume that $d=2$}.

\begin{proof}[Proof of Theorem \ref{thm main}]
Lemma \ref{lem pixz = A} proves that for $\pi^t$-a.e.\ $(x,y)$, $\text{type}(\pi_{x,y})=A(x)$. On the other hand Lemma \ref{lem pixz pitilde z} implies that for $\pi^t$-a.e.\ $(x,y)$, $\text{type}(\pi_{x,y})=D(y)$, where $D(y)$ is the common almost sure type of all $\pi_{x,y}$ which can be reached from $y$. By Fubini we have
$$\pi^t(\,\{(x,y):A(x)=D(y)\}\,)=1.$$
We want to use this to show that $A(\cdot)$ is constant on the cells of $\CCC_{\pi^t}$. We first prove this for 
$$\CCC^t:=\{\,\text{ri}\,\mbox{supp}(\pi^t_x)\}_{x\in\mathbb{R}^d}.$$ 
 By \eqref{eq Mx} and Lemma \ref{lem technical bla} (v) we have $\text{supp}(\pi^t_x)=\text{supp}(\, \law(M_t|M_0=x)\,)=\overline{\text{co}}\,\nabla F^x(\mathbb{R}^d)$ is convex. As in the final part of the proof of Lemma \ref{lem existence paving}, the martingale property implies $x\in \text{ri}\,\overline{\text{co}}\,\text{supp}\,\pi^t_x=\text{ri}\,\text{supp}\,\pi^t_x$, and by  \cite[Theorem 6.3]{Ro70} we know $\overline{\text{ri}\,\text{supp}\,\pi^t_x}=\text{supp}\,\pi^t_x$. Hence, to show that $\CCC^t$ is a candidate $\pi^t$-invariant convex paving, it remains to show that the cells of $\CCC^t$ are pairwise disjoint or equal. 

By Proposition \ref{lem essential} below, there is a $\mu$-full set of initial positions with the property that, if $x,x'$ satisfy $\text{ri}\,\text{supp}\,\pi^t_x\bigcap \text{ri}\,\text{supp}\,\pi^t_{x'}\neq \emptyset$, then $A(x)=A(x')$, i.e.\ the types of $\pi_x$ and $\pi_{x'}$ coincide. This means that $\nabla F^x$ and $\nabla F^{x'}$ are translates of each other, implying that $\overline{\text{co}}\,\nabla F^x(\mathbb{R}^d)=\overline{\text{co}}\,\nabla F^{x'}(\mathbb{R}^d)$. From the previous paragraph, this shows $\text{supp}(\pi^t_x)=\text{supp}(\pi^t_{x'})$ and in particular $\text{ri}\,\text{supp}\,\pi^t_x = \text{ri}\,\text{supp}\,\pi^t_{x'}$. In one stroke this proves that $\CCC^t$ is a ($\pi^t$-invariant) convex paving and that $A(\cdot)$ is constant on its cells.

Since $\CCC_{\pi^t}$ is finer than $\CCC^t$, this proves that $A(\cdot)$ is constant in the cells of $\CCC_{\pi^t}$ as well, and we conclude the proof by Lemma \ref{lem A constant}
\end{proof}

The crucial Proposition \ref{lem essential} below relies on the ``monotonicity principle'' of Proposition \ref{prop monotonicity general}, and more specifically Corollary \ref{prop monotonicity}. \\

\textbf{For the rest of this section let $\Gamma$ be the $\mu$-full set of Corollary \ref{prop monotonicity}.}

{
\begin{proposition}
\label{lem essential}There is a $\mu$-full set $S\subset \Gamma$ with the following property: If $x,x'\in S$  satisfy $\text{ri}\,\text{supp}\,\pi^t_x\bigcap \text{ri}\,\text{supp}\,\pi^t_{x'}\neq \emptyset$, then $A(x)=A(x')$.
\end{proposition}

\begin{proof}

\textbf{Step 1:} By Lemma \ref{lem same dimension} below, for $x,x'\in\Gamma$ we have
\begin{align}\textstyle &\text{dim}\,\text{ri}\,\text{supp}\,\pi_x^t = \text{dim}\,  \text{ri}\,\text{supp}\,\pi_{x'}^t = \text{dim}\,\bra{\text{ri}\,\text{supp}\,\pi_x^t\cap \text{ri}\,\text{supp}\,\pi_{x'}^t}\Rightarrow   \notag \\ 
 & \pi_x^t(\text{ri}\,\text{supp}\,\pi_x^t\cap \text{ri}\,\text{supp}\,\pi_{x'}^t \cap\{\pi_{x,y}\neq\pi_{x',y}\})=\pi_{x'}^t(\text{ri}\,\text{supp}\,\pi_x^t\cap \text{ri}\,\text{supp}\,\pi_{x'}^t\cap\{\pi_{x,y}\neq\pi_{x',y}\}) =0. \label{case eq dim}
\end{align}
The goal is now to prove that for pairs $x,x'\in\Gamma$ the l.h.s.\ of \eqref{case eq dim} holds.  As we will see in the final step of the proof, the r.h.s.\ of \eqref{case eq dim} is a strengthening of the dynamic programming principle that allows to deal with the null sets in Lemmas \ref{lem pixz pitilde z} and \ref{lem pixz = A} more effectively.

\textbf{Step 2:} By Lemma \ref{lem build better plan} we know that if $x,x'\in\Gamma$, then
\begin{align*}
\textstyle \text{dim}\,\text{ri}\,\text{supp}\,\pi^t_x &\textstyle =\text{dim}\, \text{ri}\,\text{supp}\,\pi^t_{x'}=1 \mbox{ and }\text{ri}\,\text{supp}\,\pi^t_x\bigcap \text{ri}\,\text{supp}\,\pi^t_{x'}\neq \emptyset \\ 
&\Rightarrow  \text{dim}\,\bra{\text{ri}\,\text{supp}\,\pi_x^t\cap \text{ri}\,\text{supp}\,\pi_{x'}^t} =1.
\end{align*} 

\textbf{Step 3:} By Lemma \ref{lem build better plan line to point}, we have for $x,x'\in\Gamma$
\begin{align*}
\text{dim}\,\text{ri}\,\text{supp}\,\pi^t_x =1\mbox{ and } \text{ri}\,\text{supp}\,\pi^t_{x'}=\{x'\} \Rightarrow x'\notin \text{ri}\,\text{supp}\,\pi^t_x\,\, .
\end{align*}

\textbf{Step 4:} By Lemma \ref{lem build better plan surface to line}, we have for $x,x'\in\Gamma$
\begin{align*}\textstyle 
\text{dim}\,\text{ri}\,\text{supp}\,\pi^t_x =2\mbox{ and } \text{dim}\,\text{ri}\,\text{supp}\,\pi^t_{x'}=1\Rightarrow \text{ri}\,\text{supp}\,\pi^t_x \bigcap 
\text{ri}\,\text{supp}\,\pi^t_{x'}=\emptyset \,\, .
\end{align*}

\textbf{Step 5:} By Lemma \ref{lem build better plan surface to cloud of dots}, the family $\CCC^t_2:=\{\text{ri}\, \text{supp}\,\pi^t_x:\, x\in\Gamma,\,\text{dim}\, \text{ri}\, \text{supp}\,\pi^t_x=2\}$ consists of pairwise either disjoint or equal sets. As these are open sets, there can only be countable many different such sets (i.e.\ $|\CCC^t_2|=|\mathbb{N}|$). If $C\in \CCC^t_2$ is such that $\mu(\{x:\text{ri}\, \text{supp}\,\pi^t_x =C\})=0$, then we discard this set $C$ from our convex paving. So we may assume, for  $C\in \CCC^t_2$ that  $\mu(\{x:\text{ri}\, \text{supp}\,\pi^t_{x} =C\})>0$. By Lemma \ref{lem build better plan surface to cloud of dots} the set $\{x'\in C: \text{ri}\, \text{supp}\,\pi^t_{x'} = \{x'\}\}$ is $\mu$-null under the assumption $\nu\ll\lambda^2$. Hence, for each of the countable many $C\in \CCC^t_2$ we can discard a $\mu$-null set such that on a possibly smaller but still $\mu$-full subset of $\Gamma$, which we keep calling $\Gamma$ for simplicity, we have
$$x,x'\in\Gamma,\, \text{dim}\, \text{ri}\, \text{supp}\,\pi^t_x=2 \,\,\mbox{ and }\,\, \{x'\} =  \text{ri}\, \text{supp}\,\pi^t_{x'} \Rightarrow x' \notin \text{ri}\, \text{supp}\,\pi^t_x.$$

\textbf{Final Step:} By Steps 3, 4 and 5, we may assume that, for $x,x'$ in a $\mu$-full set, we have
$$\textstyle\text{ri}\, \text{supp}\,\pi^t_x \bigcap 
\text{ri}\,\text{supp}\,\pi^t_{x'}\neq\emptyset\Rightarrow
\text{dim}\,\text{ri}\, \text{supp}\,\pi^t_x  = \text{dim}\,\text{ri}\, \text{supp}\,\pi^t_{x'}. $$
In this situation, if the common dimension in the r.h.s.\ is equal to one, by Step 2 also the dimension of the intersection in the l.h.s.\ is equal to one. On the other hand, if the common dimension in the r.h.s.\ is two, then automatically the dimension of the intersection is two (as an open convex set in $\mathbb{R}^2$). In any case, call $d^{(x,x')}$ this common dimension\footnote{Actually the case $d^{(x,x')}=2$ is settled by Lemma \ref{lem build better plan surface to cloud of dots}, but we prefer to give a general argument.}. We find ourselves in the setting of \eqref{case eq dim}, so by Step 1 we must have with $I:=\text{ri}\,\text{supp}\,\pi_x^t\cap \text{ri}\,\text{supp}\,\pi_{x'}^t$
\begin{align}\label{eq good set}
\pi_x^t(I \cap\{\pi_{x,y}\neq\pi_{x',y}\})=\pi_{x'}^t(I\cap\{\pi_{x,y}\neq\pi_{x',y}\}) =0.
\end{align}

Possibly throwing away another $\mu$-null set, we know by Lemma \ref{lem pixz = A} that on a $\mu$-full set $S\subset \Gamma$ and for sets 
 $\mathsf Y, \mathsf Y'$ with $\pi^t_x(\mathsf Y)=\pi^t_{x'}(\mathsf Y')=1$ it holds that
\begin{itemize}
\item[] $\text{type}(\pi_{x,y})= A(x),\, \forall y\in\mathsf Y,$
\item[] $\text{type}(\pi_{x',y})= A(x'),\, \forall y\in\mathsf Y'.$
\end{itemize}
By Lemma \ref{lem technical bla}, $\pi^t_x$ is equivalent to $d^{(x,x')}$-dimensional Lebesgue measure on the $d^{(x,x')}$-dimensional open set $\text{ri}\, \text{supp}\,\pi^t_x$. Since $\text{ri}\, \text{supp}\,\pi^t_x \bigcap 
\text{ri}\,\text{supp}\,\pi^t_{x'}$ is a $d^{(x,x')}$-dimensional open subset it is also of positive $d^{(x,x')}$-Lebesgue measure. Then it is also of positive $\pi^t_x$-measure. Thus $\text{ri}\, \text{supp}\,\pi^t_x \bigcap 
\text{ri}\,\text{supp}\,\pi^t_{x'}\bigcap \mathsf Y$ has positive  $\pi^t_x$-measure, and positive $d^{(x,x')}$-Lebesgue measure. But then again by Lemma \ref{lem technical bla} this same set must have positive $\pi^t_{x'}$-measure. We conclude that
$I\bigcap \mathsf Y\bigcap \mathsf Y'$ 
has likewise positive $\pi^t_{x'}$-measure. The symmetric argument shows that the same set has positive $\pi^t_{x}$-measure. But by \eqref{eq good set} the set $\{y:\pi_{x,y}=\pi_{x',y}\}$ is $\pi^t_x$-full in $I$. It follows that 
$$\textstyle I\bigcap \mathsf Y\bigcap \mathsf Y'\bigcap \{y:\pi_{x,y}=\pi_{x',y}\},$$
has positive $\pi^t_{x}$-measure, and by the same token it has positive $\pi^t_{x'}$-measure. In particular, 
$$\textstyle \mathsf Y\bigcap  \{y:\pi_{x,y}=\pi_{x',y}\}\bigcap \mathsf Y' \neq \emptyset,$$
and taking $y$ in this intersection we find
$$A(x)=\text{type}(\pi_{x,y})=\text{type}(\pi_{x',y})=A(x'). $$
\end{proof}

\begin{lemma}\label{lem same dimension}
We have
$$\textstyle
\left \{ (x,x'):\,
\begin{array}{c}
\text{dim}\,\text{ri}\,\text{supp}\,\pi_x^t=\text{dim}\, \text{ri}\,\text{supp}\,\pi_{x'}^t =\text{dim}\,\bra{\text{ri}\,\text{supp}\,\pi_x^t\bigcap \text{ri}\,\text{supp}\,\pi_{x'}^t},\\
\mbox{and either } \,\, \pi_x^t\left(\text{ri}\,\text{supp}\,\pi_x^t\bigcap \text{ri}\,\text{supp}\,\pi_{x'}^t \bigcap \{\pi_{x,y}\neq\pi_{x',y}\}\right) >0,\\ \mbox{or }\,\,
\pi_{x'}^t\left(\text{ri}\,\text{supp}\,\pi_x^t\bigcap \text{ri}\,\text{supp}\,\pi_{x'}^t \bigcap\{\pi_{x,y}\neq\pi_{x',y}\}\right ) >0
\end{array}
\right \}\bigcap (\Gamma\times\Gamma)=\emptyset.
$$
\end{lemma}

\begin{proof}
Take $x,x'\in\Gamma$. By Corollary \ref{prop monotonicity}, the two-step martingale $\frac{\delta_x\pi_x+\delta_{x'}\pi_{x'}}{2}$ is optimal for \eqref{prob discrete mart}. Consider its continuous-time analogue, i.e.\ the martingale which started at $x$ equals $M^x$ and  started at $x'$ equals $M^{x'}$ (cf. \eqref{eq Mx}) and both starting points have equal probability. We denote this continuous time martingale by $M^{(x,x')}$. By construction, $\law(M^{(x,x')}_t|M_0^{(x,x')}=x)=\pi^t_x$ and likewise for $x'$. Similarly $\law(M^{(x,x')}_1|M_0^{(x,x')}=x,\, M_t^{(x,x')}=y)=\pi_{x,y} $ and the same holds for $x'$ instead of $x$. By optimality of $\frac{\delta_x\pi_x+\delta_{x'}\pi_{x'}}{2}$ also  $M^{(x,x')}$ is optimal for the continuous-time analogue of \eqref{prob discrete mart}, then by dynamic programming (Lemma \ref{lem DPP}), we obtain sets $\mathsf Y, \mathsf Y'$ such that
\begin{enumerate}
\item[] $\pi_{x,y}=\law(M^{(x,x')}_1|M_t^{(x,x')}=y)$ for $y\in\mathsf Y$ with $\pi^t_x(\mathsf Y)=1$,
\item[] $\pi_{x',y}=\law(M^{(x,x')}_1|M_t^{(x,x')}=y)$ for $y\in\mathsf Y'$ with $\pi^t_{x'}(\mathsf Y')=1$.
\end{enumerate}
The important point is that this is ``pointwise'' in $M^{(x,x')}_0\in\{x,x'\}$. 

Now assume further that $\text{dim}\,\text{ri}\,\text{supp}\,\pi_x^t=\text{dim}\, \text{ri}\,\text{supp}\,\pi_{x'}^t =\text{dim}\,\bra{\text{ri}\,\text{supp}\,\pi_x^t\bigcap \text{ri}\,\text{supp}\,\pi_{x'}^t}$, and call $d^{(x,x')}$ this common dimension. By Lemma \ref{lem technical bla} we have that $\pi^t_x$ and $\pi^t_{x'}$ restricted to $\text{ri}\,\text{supp}\,\pi_x^t\bigcap \text{ri}\,\text{supp}\,\pi_{x'}^t$ are equivalent to $d^{(x,x')}$-dimensional Lebesgue measure restricted to this same set. We write 
$$\textstyle I:=\text{ri}\,\text{supp}\,\pi_x^t\bigcap \text{ri}\,\text{supp}\,\pi_{x'}^t\,.$$ 
Necessarily $\mathsf Y\bigcap I$ is $\pi^t_x$-full in $I$, and therefore also $\pi^t_{x'}$-full in $I$. But then $\textstyle \mathsf Y\bigcap I\bigcap\mathsf Y'$ is $\pi^t_{x'}$-full in $I$ too. Inverting the roles of $x,x'$ this set must also be $\pi^t_{x}$-full in $I$. We conclude 
$$ \textstyle\pi_x^t(I\backslash (\mathsf Y\bigcap\mathsf Y'))=\pi_{x'}^t(I\backslash (\mathsf Y\bigcap\mathsf Y'))=0.$$
 But on $\mathsf Y\bigcap\mathsf Y'$ we have $\pi_{x,y}=\law(M^{(x,x')}_1|M_t^{(x,x')}=y)=\pi_{x',y}$, so 
$$\textstyle \pi_x^t(I\bigcap \{\pi_{x,y}\neq\pi_{x',y}\})=\pi_{x'}^t(I\bigcap\{\pi_{x,y}\neq\pi_{x',y}\}) =0.$$ 
This concludes the proof.
\end{proof}

\begin{lemma}
\label{lem build better plan}
Put
$${\cal V}:=\{(x,x'):\text{ri}\,\text{supp}\,\pi^t_x \mbox{ and }\text{ri}\,\text{supp}\,\pi^t_{x'} \mbox{ have dimension 1 and intersect in a singleton}\}.$$
Then,
$${\cal V}\cap (\Gamma\times \Gamma)=\emptyset.$$
\end{lemma}

\begin{proof}
We use the same notation as in the proof of Lemma \ref{lem same dimension}, and assume $(x,x')\in{\cal V}$.

By construction, $\law(M^{(x,x')}_t|M_0^{(x,x')}=x)=\pi^t_x$ and likewise for $x'$. So $M^{(x,x')}_t$ conditioned to start at $x$, or at $x'$, live respectively in line segments exactly intersecting in a single point ${p}\in\mathbb{R}^2$. By Lemma \ref{lem technical bla}, the paths of these martingales (restricted to times in $[0,t]$) evolve in different ``space-time'' strips that only intersect along the line $L:=\{(p,s):s\geq 0\}$. Let $\tau:=\inf\{s: (M_s^{(x,x')},s)\in L\}$. It follows that  $0<\tau<t$ on a non-negligible set. The law of $\tau$ conditioned on the starting point of $M^{(x,x')}$ is equivalent to Lebesgue measure on $(0,1)$. The reason is that this is true for 1-dimensional Brownian motion, and thanks to Lemma \ref{lem technical bla} the martingale $M^{(x,x')}$ conditioned to start say in $x$, is a one-dimensional Brownian motion after a continuous strictly increasing time-change. Hence for any set $E\subset (0,1)$ of positive Lebesgue measure we have $\mathbb{P}(\tau\in E\bigcap (0,t)\,|\,M^{(x,x')}_0=x )>0$ and  $\mathbb{P}(\tau\in E\bigcap (0,t)\,|\,M^{(x,x')}_0=x' )>0$. Thus we observe that the law of $M_t^{(x,x')}$ given $\{M_s^{(x,x')}:\,s\leq \tau\wedge t\}$ is different from the law of $M_t^{(x,x')}$ given $M_{\tau\wedge t}^{(x,x')}$. Indeed, when $\tau<t$ (equivalently when $M_{\tau\wedge t}^{(x,x')}=p$) and $\tau\in E$, one cannot for sure say in which of the aforementioned strips the martingale will continue to evolve. On the contrary, by observing $\{M_s^{(x,x')}:\,s\leq \tau\wedge t\}$ and on $\{\tau<t\}\bigcap \{\tau\in E\}$, such a strip is completely determined. Therefore $M^{(x,x')}$ fails to have the strong Markov property. But then it cannot be optimal between its marginals, by Corollary \ref{strong Markov}, and so neither can be $\frac{\delta_x\pi_x+\delta_{x'}\pi_{x'}}{2}$ optimal for \eqref{prob discrete mart}. We conclude by Corollary \ref{prop monotonicity} that $(x,x')\notin \Gamma\times\Gamma$.
\end{proof}

\begin{lemma}
\label{lem build better plan line to point}
We have 
\begin{align}\label{eq:210}
\textstyle \{(x,x'):\text{dim}\,\text{ri}\,\text{supp}\,\pi^t_x =1,\, \text{ri}\,\text{supp}\,\pi^t_{x'} =x', \, x'\in \text{ri}\,\text{supp}\,\pi^t_x  \}\bigcap (\Gamma\times\Gamma)=\emptyset.
\end{align}
\end{lemma}

\begin{proof}
The proof is very similar to that of Lemma \ref{lem build better plan}. Let $(x,x')$ belong to the leftmost set in \eqref{eq:210}. Using the same notation, $M^{(x,x')}$ is a martingale which evolves in a space-time strip if started at $x$, and otherwise is a constant equal to $x'$. We denote $\tau$ the first hitting time of $\{(x',s):s\geq 0\}$. Since the martingale lives in a strip, we have that $\tau<t$ has probability strictly greater than $1/2$. The strong Markov property of $M^{(x,x')}$ is destroyed at $\tau\wedge t$, since the knowledge of the past up to $\tau\wedge t$ reveals whether the martingale is constant or not thereafter. As before, by Corollary \ref{strong Markov} and Corollary \ref{prop monotonicity}, $M^{(x,x')}$ cannot be optimal and $(x,x')\notin \Gamma\times\Gamma$.
\end{proof}

\begin{lemma}
\label{lem build better plan surface to line}
We have $$\textstyle \{(x,x'):\text{dim}\,\text{ri}\,\text{supp}\,\pi^t_x =2,\,  \text{dim}\,\text{ri}\,\text{supp}\,\pi^t_{x'} =1,\, \text{ri}\,\text{supp}\,\pi^t_x\bigcap \text{ri}\,\text{supp}\,\pi^t_{x'}\neq\emptyset \} \bigcap (\Gamma\times\Gamma)=\emptyset  .$$
\end{lemma}

\begin{proof}
As in the previous proofs, with $M^{(x,x')}$ we associate  $\tau=\inf\{s: M^{(x,x')}_s\in \text{ri}\,\text{supp}\,\pi^t_{x'}\}$. Taking $(x,x')$ in the leftmost set, it is tedious but not difficult to see that $$\law\left(\, (M^{(x,x')}_{\tau},\tau )\,|\tau\leq t,M^{(x,x')}_0=x\right ),\,\,\,\mbox{   and   }\,\,\, \law\left( (M^{(x,x')}_U,U)\,| M^{(x,x')}_0=x' \right),$$ are equivalent to Lebesgue measure on $\text{ri}\,\text{supp}\,\pi^t_{x'}\times [0,t]$, where $U$ is uniformly distributed on $[0,t]$ and independent of everything. The point is that there is a common ``space-time'' set $E$ charged by the two aforementioned laws. But the behaviour of $M^{(x,x')}_t$ conditioned on its past up to $\tau\wedge t$ is drastically depending on its starting position (e.g.\ whether it will evolve in a one- or two- dimensional set), whereas if for example we knew $(M^{(x,x')}_{\tau},\tau )\in E$ then this does not reveal the dimension of the set where the martingale will continue to evolve. This contradicts the strong Markov property and we conclude as before.
\end{proof}

\begin{lemma}
\label{lem build better plan surface to cloud of dots}
The family $$\CCC^t_2:=\{\text{ri}\, \text{supp}\,\pi^t_x:\,x\in \Gamma,\, \text{dim}\, \text{ri}\, \text{supp}\,\pi^t_x=2\},$$ consists of open sets which are pairwise disjoint or equal. Assuming $\nu\ll\lambda^2$, we have 
$$C\in \CCC^t_2\, \mbox{ and } \,\mu(\{x:\text{ri}\, \text{supp}\,\pi^t_{x} =C\})>0\Rightarrow \mu(\{x'\in C: \text{ri}\, \text{supp}\,\pi^t_{x'} = \{x'\}\})=0.$$ 
\end{lemma}

\begin{proof}
Let $\Lambda$ consist of all $(x,x')$ such that
$$\textstyle\{\text{dim}\, \text{ri}\, \text{supp}\,\pi^t_x = 2 = \text{dim}\, \text{ri}\, \text{supp}\,\pi^t_{x'},\, \text{ri}\, \text{supp}\,\pi^t_x\neq \text{ri}\, \text{supp}\,\pi^t_{x'}, \text{ri}\, \text{supp}\,\pi^t_x \bigcap \text{ri}\, \text{supp}\,\pi^t_{x'} \neq \emptyset \}.$$
As before we can show that $\Lambda$ cannot intersect $\Gamma\times\Gamma$. We do not give the argument, to avoid repetition, but mention that instead of contradicting the strong Markov property it suffices to contradict the regular Markov property. We conclude the first assertion. 

\comment{JB:for this part it seems we could have assumed $\mu\ll\lambda^2$ instead.}
Now let $C\in \CCC^t_2$ such that $\mu(\{x:\text{ri}\, \text{supp}\,\pi^t_{x} =C\})>0$, $K:=\{x'\in C: \text{ri}\, \text{supp}\,\pi^t_{x'}=\{x'\}\}$, and suppose $\mu(K)>0$. We think of $K$ as a non-negligible cloud of dots where the martingale $M$ stays frozen. Since $M_0\in K\Rightarrow M_1\in K$, we have $v(K)>0$ and by assumption $\lambda^2(K)>0$. It follows that $\{M_t\in K\}$ is non-negligible, no matter if $M$ has started on $K$ or on $\{x:\text{ri}\, \text{supp}\,\pi^t_{x} =C\} $ at time zero (in the latter case, by Lemma \ref{lem technical bla}). Since both sets of initial conditions are non-negligible, we contradict the regular Markov property of $M$. Indeed, on $\{M_t\in K\}$ the behaviour of $M$ after $t$ is drastically different depending on the starting condition at time zero being in $K$ or $\{x:\text{ri}\, \text{supp}\,\pi^t_{x} =C\} $. This contradicts the optimality of $M$, and we conclude $\mu(K)=0$.
\end{proof}

\section{Further optimality properties}\label{sec causal}

Let $\mathbb{T}:=\{0=t_0\leq t_1\leq \dots\leq t_{n-1}\leq t_n=1\}\subset [0,1]$ be a finite subgrid. Suppose $M$ is a standard stretched Brownian motion from $\mu$ to $\nu$, so $M_t=f_t(B_t)$ for a Brownian motion starting with some distribution $\alpha$; see \eqref{eq f_t}. Then
\begin{align*}
M_{t_0}= f_0(B_{t_0}),\,\,
M_{t_1}= f_1(B_{t_1}),\,\, \dots \,\, ,
\,\,
M_{t_n}= f_n(B_{t_n}).
\end{align*}
Denote $\nu^{\mathbb{T}}:=\law(M_{t_0},M_{t_1},\dots,M_{t_n})$, the projection of $\nu$ onto the time indices in $\mathbb{T}$, and $\gamma^{\mathbb{T}}:=\law(B_{t_0},B_{t_1},\dots,B_{t_n}) $. Finally, consider the \textit{adapted map}
$$ [\mathbb{R}^d]^{n+1}\ni(b_0,\dots, b_{n})\mapsto f^{\mathbb{T}}(b_0,\dots, b_{n}):=(f_0(b_0),f_1(b_1),\dots, f_n(b_n))\in [\mathbb{R}^d]^{n+1} .$$
It follows that $$f^{\mathbb{T}}(\gamma^{\mathbb{T}})=\nu^{\mathbb{T}}.$$
Each component of $f^\mathbb{T}$ is \textit{increasing} in the sense that it is the gradient of a convex function. Such a map is an example of an ``increasing triangular transformation,'' as in \cite{BoKoMe05}. It can also be understood in terms of increasingness w.r.t.\ lexicographical order in case $\nu^\mathbb{T}$ has a density. In a sense properly explained in \cite{BaBeLiZa16}, $f^{\mathbb{T}}$ sends  $\gamma^{\mathbb{T}}$ into $\nu^{\mathbb{T}}$ in a canonical respect.\ optimal way: see respect.\ Proposition 5.6 and Corollary 2.10 therein.

Since this is true no matter the subgrid $\mathbb{T}$, we are entitled to think of $M$ as an \textit{adapted increasing rearrangement} of the Brownian motion into a martingale with given initial an final laws. Also, the aforementioned canonical/optimal character of such rearrangements should translate into the optimality of $M$ as obtained in the previous section, and vice-versa. We now make this heuristics rigorous.

Problem \eqref{MBMBB} is equivalent to 
\begin{align}\label{eq starting point}
\inf_{M_t=M_0+\int_0^t \sigma_s\, dB_s, M_0\sim \mu,M_1\sim \nu} {\mathbb E}\left[\text{tr}\,\langle M-B\rangle_1\right],
\end{align}
since ${\mathbb E}\left[ \text{tr} \,\langle M-B\rangle_1\right] = {\mathbb E}\left[\text{tr}\,\langle M\rangle_1\right] + {\mathbb E}\left[\text{tr}\,\langle B\rangle_1\right]-2 {\mathbb E}\left[  \int_0^1 \text{tr}(\sigma_t)dt\right]$, and the first two quantities in the r.h.s.\ do not depend on the concrete coupling $(M,B)$. This also proves that for \eqref{eq starting point} it is irrelevant where $B$ is started. We now want to formulate a transport problem between laws of stochastic processes which is compatible with \eqref{eq starting point}. For ease of notation we denote $$\Omega:= C([0,1];\mathbb{R}^d).$$

\begin{definition}
A causal coupling between $\mathbb{P}$ and $\mathbb{Q}$ is a probability measure $\pi$ on $\Omega\times\Omega$ with first and second marginals $\mathbb{P}$ and $\mathbb{Q}$ respectively, and satisfying the additional property
\begin{align}\label{eq causality}
\forall t,\forall A\in \mathcal{F}_t:\,\, \left ( \,\, \Omega\ni x\mapsto\pi_x(A)\in [0,1]\,\, \right ) \mbox{ is $(\mathbb{P},\mathcal{F}_t)$-measurable},
\end{align}
where $\mathcal{F}$ is the $\mathbb{P}$-completed canonical filtration and $\pi_x$ is a regular conditional probability of $\pi$ w.r.t.\ the first marginal. We denote $\Pi_c(\mathbb{P},\mathbb{Q})$ the set of all such $\pi$. We also denote $\Pi_{bc}(\mathbb{P},\mathbb{Q})=\{\pi \in \Pi_c(\mathbb{P},\mathbb{Q}):\, e(\pi)\in \Pi_c(\mathbb{Q},\mathbb{P}) \}$  for $e(x,y)=(y,x)$, the set of bicausal couplings.
\end{definition}

We refer to \cite{La13,BaBeLiZa16,AcBaZa16,BaBeEdPi17} for more on this definition. In what follows, we write $(\omega,\bar{\omega})$ for a generic element in $\Omega\times\Omega$.

\begin{lemma}\label{lem joint mart}
Let $\mathbb{P}$ and $\mathbb{Q}$ be martingale laws, and $\pi \in \Pi_{bc}(\mathbb{P},\mathbb{Q})$. Then the canonical process on  $\Omega\times\Omega$ is a $\pi$-martingale in its own filtration.
\end{lemma}

\begin{proof}
One can easily see that under $\pi$ we have
\begin{align*}
\{\bar{\omega}_s: 0\leq s\leq t\} \mbox{ is $\pi$-conditionally independent from }\{\omega_s: 0\leq s\leq 1\} \mbox{ given }\{\omega_s: 0\leq s\leq t\},\\
\{\omega_s: 0\leq s\leq t\} \mbox{ is $\pi$-conditionally independent from }\{\bar{\omega}_s: 0\leq s\leq 1\} \mbox{ given }\{\bar{\omega}_s: 0\leq s\leq t\},
\end{align*}
by bicausality. The first property above, for $T>t$, implies
$$\mathbb{E}^\pi[\omega_T|\,\{ \omega_s,\bar{\omega}_s,s\leq t\}\,] = \mathbb{E}^\pi[\omega_T|\,\{ \omega_s,s\leq t\}\,] = \mathbb{E}^{\mathbb{P}}[\omega_T|\,\{ \omega_s,s\leq t\}\,]=\omega_t.$$
The second property implies similarly $\mathbb{E}^\pi[\bar{\omega}_T|\,\{ \omega_s,\bar{\omega}_s,s\leq t\}\,]=\bar{\omega}_t$, so we conclude.
\end{proof}

Let us denote by $\mathbb{W}$ Wiener measure (started at zero) on $\Omega$. The next lemma establishes the crucial connection between standard stretched Brownian motion and the present \emph{causal transport} setting.

\begin{lemma}\label{lem identification causal}
Let $M$ be standard stretched Brownian motion from $\mu$ to $\nu$, with $M_t=M_0+\int_0^t\sigma_s dB_s$. Then $$\law(B-B_0,M)\in \Pi_{bc}(\mathbb{W},\law(M)).$$
More generally, if $M$ is stretched Brownian motion and $B$ is as in Remark \ref{rem connection}, the same conclusion holds.
\end{lemma}

\begin{proof}
Let $M$ be standard stretched Brownian motion from $\mu$ to $\nu$. By Lemma \ref{lem technical bla} there is an orthogonal projection $P$ such that $M_t=\tilde f_t(\tilde B_t)$, where $\tilde B_t =PB_t$. By the same result, the filtrations of $M$ and $\tilde B$ coincide. This shows that the coupling $\law(B-B_0,M)$ is causal from $\mathbb W$ to $\law(M)$. For the reverse causality, it suffices to  observe that $\{\tilde B_{t+h}-\tilde B_t:h\geq 0\}$ is independent from $\{B_s:s\leq t\}$, so in particular given $\{M_s:s\leq t\}$ we have that $\{B_s-B_0:s\leq t\}$ and $\{M_s:s\leq 1\}$ are independent. The case of $M$ a stretched Brownian motion is similar, taking $B$ independent of $M_0$ and upon conditioning on the latter random variable.
\end{proof}

We can now put the pieces together to obtain optimality of (standard) stretched Brownian motion in the sense of trajectorial laws. Let us fix a refining sequence of partition $P_n $ of $[0,1]$ in order to define the quadratic variation $\langle \,\cdot\,\rangle$ pathwise on $C([0,1];\mathbb{R}^d)$ in the usual manner, namely
$$\omega\mapsto \langle \omega \rangle_1^{i,j}:= \lim_{n\to\infty} \sum_{t_m\in P_n} (\omega^i_{t_{m+1}} -\omega^i_{t_m})(\omega^j_{t_{m+1}} -\omega^j_{t_m}) , $$
when the limit exist, and otherwise $+\infty$. We then consider

\begin{align}\label{eq trajectorial}
\inf_{\substack{\mathbb{Q} \in \M^c (\mu,\nu) \\ \pi\in \Pi_{bc}(\mathbb{W},\mathbb{Q})}}\mathbb{E}^{\pi}[\text{tr}\,\langle \omega-\bar{\omega} \rangle_1 ],
\end{align}
where $\M^c (\mu,\nu)$ denotes the set of laws of continuous martingales indexed by $[0,1]$ starting in $\mu$ and terminating in $\nu$. 

\begin{proposition}\label{prop causal}
Problems \eqref{eq starting point} and \eqref{eq trajectorial} are equivalent. In particular, let $M^*$ be the optimizer of the former, i.e.\ stretched Brownian motion. Then $\QQ^*:=\law(M^*)$ is optimal for the latter.
\end{proposition}

\begin{proof}
Let $\mathbb{Q},\pi$ be feasible for \eqref{eq trajectorial}. Since
\begin{align*}
\mathbb{E}^{\pi}[ \text{tr}\,\langle \omega-\bar{\omega} \rangle_1 ]&= \mathbb{E}^{\mathbb{W}}[\text{tr}\,\langle \omega \rangle_1 ] + \mathbb{E}^{\mathbb{Q}}[\text{tr}\,\langle \bar{\omega} \rangle_1 ]-2\mathbb{E}^{\pi}[\text{tr}\,\langle \omega,\bar{\omega} \rangle_1 ]\\
&= \mathbb{E}^{\mathbb{W}}[|\omega_1|^2-|\omega_0|^2 ] + \mathbb{E}^{\mathbb{Q}}[|\bar{\omega}_1|^2 - |\bar{\omega}_0|^2 ]-2\mathbb{E}^{\pi}[\text{tr}\,\langle \omega,\bar{\omega} \rangle_1 ],
\end{align*}
we can equivalently maximize $\mathbb{E}^{\pi}[\text{tr}\,\langle \omega,\bar{\omega} \rangle_1 ]$ in \eqref{eq trajectorial}, rather than minimizing $\mathbb{E}^{\pi}[\text{tr}\,\langle \omega-\bar{\omega} \rangle_1 ]$. However by Lemma \ref{lem joint mart} the canonical process is a $\pi$-martingale so
$$\mathbb{E}^{\pi}[\text{tr}\,\langle \omega,\bar{\omega} \rangle_1 ] = \mathbb{E}^{\pi}[ \omega_1 \cdot \bar{\omega}_1]= \mathbb{E}^{\pi}[ \, \mathbb{E}^{\pi}[\omega_1 \cdot \bar{\omega}_1 \,|\, \bar{\omega}_0]\,] ,$$
by the product formula and as $\omega_0=0$ under $\pi$. Denoting $\pi_x=\law_\mathbb{Q}(\bar{\omega_1}|\bar{\omega}_0=x)$ and $q_x=\law_\pi((\bar{\omega}_1,\omega_1)|\bar{\omega}_0=x)$ we have that the first marginal of $q_x$ is $\pi_x$ and the second one is $\gamma^d$. Indeed, by bicausality $\pi-\law(\omega_1|\omega_0,\bar{\omega}_0)=\pi-\law(\omega_1|\omega_0)=\gamma^d$, so in particular $\pi-\law(\omega_1|\bar{\omega}_0)=\gamma^d$. Therefore
\begin{align}\label{eq chain ineq}
\mathbb{E}^{\pi}[\text{tr}\,\langle \omega,\bar{\omega} \rangle_1 ] = \int \mu(dx)\int q^x(dm,db)\, m\cdot b \leq \int \mu(dx)\sup_{q\in \Pi(\pi_x,\gamma^d) }\int q(dm,db)\, m\cdot b. 
\end{align}
By Theorem \ref{lem inequality static dynamic} we conclude that the value of \eqref{eq starting point} is greater or equal than that of \eqref{eq trajectorial}. Let $M^*$ be the optimizer of \eqref{eq starting point} (equiv.\ of \eqref{MBMBB}). By Remark \ref{rem connection} $M^*$ is precisely built via attaining the r.h.s.\ of \eqref{eq chain ineq} when maximizing over kernels $\pi^x$. By the final part of Lemma \ref{lem identification causal} we may build a bicausal coupling $\pi$ so that in \eqref{eq chain ineq} we have equality. This proves that Problems \eqref{eq starting point} and  \eqref{eq trajectorial} have the same value and that $\law(M^*)$ is optimal  for the latter.
\end{proof}
\begin{remark}
The discrete-time version of Problem \eqref{eq trajectorial} would have shown, in light of \cite{BaBeLiZa16}, that the optimal way to send a Gaussian random walk into a martingale is through the Knothe-Rosenblatt rearrangement (the unique increasing bicausal triangular transformation between its marginals). This is in tandem with the first paragraphs of the present part (once we switched to increments $b_i-b_{i-1}$). Via Proposition \ref{prop causal} we know that stretched Brownian motion attains Problem \eqref{eq trajectorial}. Hence, one can arguably describe stretched Brownian motion as the canonical/optimal Knothe-Rosenblatt rearrangement of Brownian motion with prescribed initial and final marginals.
\end{remark}

\section*{Acknowledgements} We thank Dario Trevisan for stimulating discussions at the outset of this project.

\bibliography{joint_biblio}{}
\bibliographystyle{plain}
\end{document}